\documentclass[11pt]{amsart}

\usepackage[a4paper,centering]{geometry}
\usepackage{color}
\usepackage{amssymb}
\usepackage{enumerate}
\usepackage{graphicx,psfrag}

\theoremstyle{plain}% default
\newtheorem{theorem}{Theorem}
\newtheorem{lemma}{Lemma}
\newtheorem{proposition}{Proposition}
\newtheorem{corollary}{Corollary}

\newtheorem{question}{Question}

\theoremstyle{definition}

\newtheorem{conjecture}{Conjecture}
\newtheorem{example}{Example}

\theoremstyle{remark}
\newtheorem{remark}{Remark}

\newtheorem*{problem*}{Problem}

\usepackage{parskip}

\newcommand{\R}{\mathbb R}

\title[On the number of vertices of projective polytopes]{On the number of vertices of projective polytopes}

\thanks{$^1$ Supported by  LAISLA (CONACYT-CNRS-UNAM).$^3$ Partially supported by Program MATH AmSud, Grant MATHAMSUD 18-MATH-01, Project FLaNASAGraTA and by INSMI-CNRS}

\date{\today}

\author[N. Garc\'{i}a-Col\'{i}n]{Natalia Garc\'{i}a-Col\'{i}n $^1$}
\address{$^1$ CONACYT Research Fellow - INFOTEC Centro de Investigación en Tecnologías de la Información y Comunicación, Mexico}
\email{ natalia.garcia@infotec.mx}
\author[L.P. Montejano]{Luis Pedro Montejano $^2$}
\address{$^2$ Serra H\'unter Fellow, Universitat Rovira i Virgili, Departament d'Enginyeria Inform\`{a}tica i Matem\`{a}tiques\\ Av. Pa\"{i}sos Catalans 26, 43007 Tarragona, Spain}
\email{ luispedro.montejano@urv.cat}
\author[J.L. Ram\'{i}rez Alfons\'{i}n]{Jorge Luis Ram\'{i}rez Alfons\'{i}n $^3$ }
\address{$^3$ IMAG, Univ. Montpellier, CNRS, Montpellier, France and UMI2924 - Jean-Christophe Yoccoz, CNRS-IMPA}
\email{ jorge.ramirez-alfonsin@umontpellier.fr}

\begin{document}
%%%%%%%%%% ABSTRACT %%%%%%%%%%%%
\begin{abstract}
Let $X$ be a configuration of $n$ points in $\R^d$.

What is the maximum number of vertices that $conv(T(X))$ can have among all the possible permissible projective transformations $T$?

 In this paper, we investigate this and connected questions. After presenting several upper bounds, we study a closely related problem (via Gale transforms) concerning the number of minimal Radon partitions of a set of points. We then present some bounds for this number that enable us to partially answer a question due to Pach and Szegedy. We also discuss another related problem concerning the size of topes in arrangements of hyperplanes.
\end{abstract}

\maketitle
%%%%%%%%%% INTRODUCTION %%%%%%%%%%%%
\section{Introduction}

Consider the following question

\begin{quote}
Given a set of $n$ points in general position $X \subset \R^d$, what is the maximum number of $k$-faces that $conv(T(X))$ can have among all the possible {\em permissible projective transformations} $T$?
\end{quote}

More precisely, recall that a \emph{projective transformation} $T:\mathbb{R}^{d} \rightarrow \mathbb{R}^{d}$ is a function such that $T(x)=\frac{Ax+b}{\langle c,x \rangle + \delta}$, where $A$ is a linear transformation of $\mathbb{R}^{d}$, $b, c \in \mathbb{R}^{d}$ and $\delta \in \mathbb{R},$ is such that at least one of $c\neq 0$ or $\delta \neq 0.$  $T$ is said to be \emph{permissible} for a set $X\subset \mathbb{R}^{d}$ if and only if  $\langle c,x \rangle + \delta \neq 0$ for all $x\in X$.
%$P$ is \emph{non-singular} iff the matrix $T'=\left(\begin{array}{cc} A & b^{t}\\ c & \delta \\ \end{array} \right)$ is non-singular.

Let $d\ge k\ge 0$ be integers and let $X\subset \R^d$ be a set of points in general position, we define the number of \emph{projective $k$--faces} of $X$ as
\begin{equation}\label{eq:def-h}
h_k(X,d)= \max\limits_T \left\{f_k(conv(T(X)))\right\},
\end{equation}
where the maximum is taken over all possible permissible projective transformations $T$ of $X$ and $f_k(P)$ denotes the number of $k$-faces of a polytope $P$.

Now we can define the function $H_k(n,d)$ which determines the \emph{maximum} number of {\em projective $k$-faces} that \emph{any} $X$ configuration of $n$ points in $\R^d$ must have as,
$$H_k(n,d)=\min\limits_{X \subset \R^d, |X|=n}\left\{h_k(X,d)\right\}.$$

In this paper, we focus our attention on the behavior of $H_{0}(n,d)$ (the number of projective vertices).
It turns out that $H_{0}(n,d)$ is the source of several applications.

\subsection{Scope/general interest}
The function $H_{0}(n,d)$ is closely connected with different notions/problems : McMullen's problem, bounds for $H_{d-1}(n,d)$ and its connection with minimal Radon partitions, tolerance of finite sets and arrangements of hyperplanes.

\subsubsection{McMullen's problem}$H_{0}(n,d)$ is a natural generalization of the following well-known problem of McMullen \cite{LARMAN72}:

\emph{What is the largest integer $\nu(d)$ such that any set of $\nu(d)$ points in general position, $X \subset \R^{d},$ can de mapped by a permissible projective transformation onto the vertices of a convex polytope?}

The best known bounds for McMullen's problem are:
 \begin{equation}\label{eq:boundsmcmullen}
 2d+1 \leq \nu(d) < 2d + \left\lceil \frac{d+1}{2} \right\rceil.
 \end{equation}

The lower bound was given by Larman \cite{LARMAN72} while the upper bound was provided by Ramirez Alfonsin \cite{RAMIREZALFONSIN2001}.
In the same spirit, the following function has also been investigated:

\begin{quote}
$\nu(k,d):=$ the largest integer $m$ such that any set of $n$ points in general position in $\R^d$ can be mapped, by a permissible projective transformation, onto the vertices of a $k$--neighborly polytope.
\end{quote}

Garc\'{i}a-Col\'{i}n \cite{NGC2015} proved that, for  each $3 \leq k \leq \left\lfloor \frac{d}{2}\right\rfloor$

 \begin{equation}\label{eq:boundsmcmullengen}
 d + \left\lceil \frac{d}{k} \right\rceil +1 \leq \nu(k,d) < 2d-k+1,
 \end{equation}

and that
 \begin{equation}\label{eq:boundsmcmullengen1}
 d + \left\lceil \frac{d}{2} \right\rceil +1 \leq \nu(2,d) < 2d+1.
 \end{equation}
These inequalities will be useful later for our proposes.

Let $t\ge 0$ be an integer. Let us define the following function.

\begin{quote}
$n(t,d):=$ the largest integer $n$ such that any set of $n$  points in general position in $\R^d$ can be mapped, by a permissible projective transformation onto the vertices of a convex polytope with at most $t$ points in its interior.
\end{quote}

The function $n(t,d)$ will allow us to study $H_0(n,d)$ in a more general setting, that of oriented matroids.

\begin{remark}\label{relationf_0andn(d,t)}
We have that $n(0,d)=\nu(d)$ and $H_0(n(d,t),d)=n(t,d)-t$.
\end{remark}

Our first main contribution is the following

\begin{theorem}\label{boundsforf_o} Let $d,l\ge 1$ and $n\ge 2$ be integers. Then,
{\scriptsize $$H_0(n,d)\left\{\begin{array}{ll}
=2 & \text{ if } d=1,\ n\ge 2,\\
=5 & \text{ if } d=2, \  n\ge 5, \\
\le 7 & \text{ if } d=3, \ n\ge 7,\\
=n & \text{ if } d\ge 2, \ n\le 2d+1,\\
\le n-1& \text{ if } d\ge 4, \ n\ge 2d+\lceil\frac{d+1}{2}\rceil,\\
\le n-2 & \text{ if } d\ge 4, \ n\ge  2d+\lceil\frac{d+1}{2}\rceil+1,\\
\le n-(l+2) & \text{ if } d\ge 4, \ {2d+3+l(d-2) \leq n < 2d+3+(l+1)(d-2)}.\\
%\textcolor{blue}{\le n-(l+2) }& \textcolor{blue}{\text{ if }  d\ge 4, \ d\text{-even},\ {2d+3+ld/2 \leq n < 2d+3+(l+1)d/2}.}\\
\end{array}\right.$$}
\end{theorem}

A straightforward consequence of Theorem \ref{boundsforf_o} is the following

\begin{corollary}\label{boundsforf_d-1} Let $d,l\ge 1$ and $n\ge 2$ be integers. Then,
{\scriptsize $$H_{d-1}(n,d)\left\{\begin{array}{ll}
=2 & \text{ if } d=1, \ n\ge 2,\\
=5 & \text{ if } d=2, \ n\ge 5, \\
\le 10 & \text{ if } d=3, \ n\ge 7,\\
\le f_{d-1}(C_d(n)) & \text{ if } d\ge 2, \ n\le 2d+1,\\
\le f_{d-1}(C_d(n-1))& \text{ if } d\ge 4, \ n\ge 2d+\lceil\frac{d+1}{2}\rceil,\\
\le f_{d-1}(C_d(n-2)) & \text{ if } d\ge 4, \ n\ge  2d+\lceil\frac{d+1}{2}\rceil+1,\\
\le f_{d-1}(C_d(n-l-2)) & \text{ if } d\ge 4, \ {2d+3+l(d-2) \leq n < 2d+3+(l+1)(d-2)}\\
%\textcolor{blue}{\le f_{d-1}(C_d(n-l-2)) }& \textcolor{blue}{\text{ if } d\ge 4, \ d\text{-even},\ {2d+3+ld/2 \leq n < 2d+3+(l+1)d/2}, \ l \ge 1.}\\
\end{array}\right.$$}
where $C_d(n)$ denotes the $d$--dimensional {\em cyclic polytope} with $n$ vertices.

Moreover,  $H_{d-1}(n,d)\ge n(d-1)-(d+1)(d-2)$ when $d\ge 2,  n\le 2d+1$.
\end{corollary}

\subsubsection{Minimal Radon partitions}  Recall that given a set of points in general position $X=\{x_1,\dots,x_n\}\subset\R^d,$ with $n\geq d+2$,  $A, B$ is a {\em Radon partition} of $X$ if $X=A\cup B$, $A\cap B=\emptyset$, and $conv(A) \cap conv(B)\neq \emptyset$.

It happens that $H_{d-1}(n,d)$ is very useful to count minimal Radon partitions. More specifically,
let $X={A}\cup {B}$ be any partition of $X$, we define $r_{X}({A},{B})$ as the number of $(d+2)$--size subsets $S \subset X$ such that $conv( A\cap S) \cap conv({B} \cap S) \neq \emptyset$,
that is, as the number of {\em minimal (size) Radon partitions} induced by ${A}$ and ${B}$.

We define the functions
$$r(X):= \max\limits_{ \{ ({A}, {B}) | A\cup {B} = X \}} r_{X}({A}, {B}) \:\:\:\: \text{  and  }\:\:\:\: r(n,d):=\min_{ X \subset \mathbb{R}^d, |X|=n} r(X).$$

Our second main result establishes the connection with $H_{d-1}(n,d)$.

\begin{theorem}\label{lem:radon}  Let $d,n\ge 1$ be integers. Then, $$r(n,d)=H_{d'-1}(n,d') \text{ where } d'=n-d-2.$$
\end{theorem}

We shall prove this by using the duality between {\em Gale transforms} and projective transformations. Theorem \ref{lem:radon} will play a central role in the study of a problem due to Pach and Szegedy \cite{PACH2003}.

\subsubsection{Pach and Szegedy's question}  In \cite{PACH2003}, Pach and Szegedy investigated the probability that  a triangle induced by 3 randomly  and  independently selected  points in the plane contains the origin in its interior. They remarked  \cite[Last paragraph]{PACH2003}  that in order to generalize their arguments to 3-space the following problem should be solved.

\begin{question}\label{question:ps} Given $n$ points in general position in the plane, coloured red and blue, maximize the number of multicoloured 4-tuples with the property that the convex hull of its red elements and the convex hull of its blue elements have at least one point in common. In particular, show that when the maximum is attained, the number of red and blue elements are roughly the same.
\end{question}

This question may be studied in any dimension. However, if the dimension and the number of points are very similar then optimal partitions can be unbalanced. For example, one may consider $d+2$ points in $\mathbb{R}^d$ with one point contained in the simplex spanned by the remaining $d+1$ points. The optimal partition will have $1$ red point and $d+1$ blue points and becomes arbitrarily unbalanced as $d$ goes to infinity. Nevertheless, it is not clear whether for a large set of points with respect to the dimension it is also possible that very unbalanced partitions optimize the maximum number of induced Radon partitions.

%Given a partition $A,B$ of a set of points $X$ in the plane, an \emph{induced minimal Radon partition} of $X$ is a subset $Y\subset X$ such that $|Y|=4$ and $conv(A\cap Y) \cap conv(B\cap Y) \neq \emptyset.$  Using this terminology we can

Our third main contribution provides an answer to Question \ref{question:ps} {in the case when $n$ is odd}.

\begin{theorem}\label{thm:unbalanced}  Let $X\subset \mathbb{R}^2$ be a set of points in general position with $|X|=n\geq 9$. Then, $r(n,2)$ is of order $o(n^4)$. Moreover, $X$ always admits a partition $A, B$ with $|A|<\lfloor \frac{n}{2}\rfloor+2$ and $|B|<\lfloor \frac{n}{2}\rfloor+2$ such that $r_X(A,B)=r(n,2)$ is of order $o(n^4)$.

Furthermore, in the case when $|X|$ is odd, if $A, B$ is a partition such $r_X(A,B)=r(n,2)$ then we still have $|A|<\lfloor \frac{n}{2}\rfloor+2$ and $|B|<\lfloor \frac{n}{2}\rfloor+2$.
\end{theorem}

 %We will later use the bounds on $H_{d-1}(n,d)$ with the appropriate inputs to provide bounds for the maximum number of induced Radon partitions.

\subsubsection{Tolerance}
Let us define the following function

\begin{quote}
$\lambda(t,d)$:= the   smallest number $\lambda$ such that for any set $X$ of  $\lambda$ points in
  $\mathbb{R}^d$ there exists a partition of $X$ into two sets $A$,$B$ and a subset $P\subseteq X$ of cardinality  $\mu-i$, for some $0\le i\le t$, such that $conv(A\setminus y)\cap conv(B\setminus y)\neq\emptyset$ for every  $y\in P$ and $conv(A\setminus y)\cap conv(B\setminus y)=\emptyset$ for every  $y\in X\setminus P$.
\end{quote}

There is an atractive relationship between $n(t,d)$ and $\lambda(t,d)$. The latter is shown in Subsection \ref{sec:tolerence} by using Gale's diagrams (Proposition \ref{prop:nlambda}).

The parameter $\lambda(t,d)$ can be considered as a generalization of the {\em tolerant} Radon theorem stating that there is a minimal positive integer $N=N(t,d)$ so that any set $X\subset\mathbb{R}^d$ with $|X|=N$ allows a partition into two pairwise disjoint subsets $X=A\cup B$ such that after deleting any $t$ points from $X$ the convex hulls of remaining parts intersect, i.e.,

$$conv(A\setminus Y)\cap conv(B\setminus Y)\neq\emptyset \text{ for any } Y\subset X, \ |Y|=t.$$

The information on $\lambda(d,t)$ sheds light on the understanding of the tolerant Radon theorem and a more general version known as the {\em tolerant Tverberg theorem}, see \cite{GarciaRaggiPensado,SoberonStrauz}.

\subsubsection{Arrangements of (pseudo)hyperplanes} A {\em projective} $d$-arrangement of $n$ pseudo-hyperplanes $\mathcal{H}(d,n)$ is a finite collection of pseudo-hyperplanes in the projective space $\mathbb{P}^d$ such that no point belongs to every hyperplane of $\mathcal{H}(d,n)$. Any such arrangement, $\mathcal{H}$ decomposes $\mathbb{P}^d$  into a $d$--dimensional cell complex. A cell of dimension $d$ is usually called a \emph{tope} of the arrangement $\mathcal{H}$. The \emph{size} of a tope is the number of pseudo-hyperplanes bordering it.

%\textcolor{red}{Consider the following question
%\begin{quote} Given a simple arrangement $H$ of $n$ (pseudo)hyperplanes in $\mathbb{P}^d$, what is the largest size of a tope that $H$ can have?
%\end{quote}}

A classic research topic is to study the combinatorics of the topes in arrangements of hyperplanes. For instance, it is known \cite{Roudneff1988, Shannon1979} that arrangements of $n$ hyperplanes (that is, realizable oriented matroids) always admit $n$ topes of size $d+1$ (a simplex). In \cite{Richter}, Richter proved that the number of topes simplices in an arrangement of $4k$ pseudo hyperplanes  in $\mathbb{P}^3$ is at most $3k+1$ for $k\ge 2$. Finding a sharp lower bound for the number of simplices in the non-realizable case is an open problem for $d\ge 3$. Las Vergnas conjectured that in fact every arrangement of (pseudo) hyperplanes in $\mathbb{P}^d$ admits at least one simplex. In \cite{Roudneff1991}, Roudneff proved that the number of {\em complete topes} (a tope touching all the hyperplanes) of the {\em cyclic arrangement} on dimension $d$ with $n$ hyperplanes, is at least
$\sum\limits_{i=0}^{d-2} {{n-1}\choose{i}}$ and conjectured \cite[Conjecture 2.2]{Roudneff1991} that for every $d$-arrangement  of $n>2d+ 1>5$  (pseudo)hyperplanes has at most this number of complete topes; see \cite{MonteRam} for the proof of this conjecture for an infinite family of arrangements.
\smallskip

It happens that the function $H_0(n,d)$ is very helpful to investigate the size's behavior of topes in arrangements of (pseudo)hyperplanes. Here, we may consider the following questions :
\begin{quote} Are there simple arrangements of $n$ (pseudo)hyperplanes in $\mathbb{P}^d$ in which every tope is of at most {\em certain size} ?
\end{quote}
\begin{quote} Which arrangements of $n$ (pseudo)hyperplanes in $\mathbb{P}^d$ contain a tope of at least {\em certain size} ?
\end{quote}

We partially answer these questions for small values of $d$.

\subsection{Paper's organization}
The structure of the paper is as follows: in next section we give some easy values and bounds for both $H_{0}(n,d)$ and $H_{d-1}(n,d)$ (Propositions \ref{easybounds} and \ref{easybounds}).

In Section \ref{sec:lawrence}, we discuss the treatment of the function $n(d,t)$ in the oriented matroid setting. We also recall several notions and results on oriented matroids and, specifically, on the special class of {\em Lawrence oriented matroids} (LOM) that are needed for the rest of the paper.

In Section \ref{cotassupmatroides}, we present several upper bounds based on specific constructions of LOM (Theorems \ref{dimension2}, \ref{dimension3}, \ref{cotageneral}, \ref{tpequeno}). The latter yield to the proofs of Theorem \ref{boundsforf_o} and Corollary \ref{boundsforf_d-1} also presented in this section.

After recalling the relationship between Gale transforms and projective transformations, we prove Theorem \ref{lem:radon} in Section \ref{sec:RadonPart}. We also present values and bounds for $r(n,2)$ (Theorem \ref{equalityd=2}) that we use to prove Theorem \ref{thm:unbalanced} at the end of this section.

In Section \ref{sec:concluding-rem}, we present some results concerning the size of {\em topes} in arrangements of (pseudo)hyperplanes. Finally, in an Annex, we present the proof of a result (Theorem \ref{dpar}) improving the upper bound (given in Theorem \ref{cotageneral}) when $d$ is even.

%%%%%%%%%%%%%%%%%%%%%%%%%%%%%%%%%%%%
%%%%%%%%%%%%%%%%%%%%%%%%%%%%%%%%%%%%
\section{Some basic results}\label{sec:basic}

The well-known Upper Bound Theorem  (UBT) \cite{MCMULLEN1971187} states that for all $1\leq k \leq d$,
$$f_{k-1}(P) \leq f_{k-1}(C_d(n))$$
among all simplicial (convex) polytopes with $n$ vertices $P \subset \R^d$  where $C_d(n)$ is the $d$--dimensional {\em cyclic polytope} with $n$ vertices, usually defined as the convex hull of $n$ distinct points in the moment curve $x(t):=(t,t^2,\dots ,t^d)$.

For $d\ge 2$ and $0\le k\le d-1$, the number of $k$-faces of $C_d(n)$ with $n$ vertices is given by

\begin{eqnarray}\label{form-cyclic}
f_{k}(C_{d}(n))= \displaystyle \frac{n-\delta(n-k-2)}{n-k-1}\sum\limits_{j=0}^{\lfloor\frac{d}{2}\rfloor}
\binom{n-1-j}{k+1-j} + \displaystyle\binom{n-k-1}{2j-k-1+\delta}
\end{eqnarray}

where $\delta=d-2\lfloor\frac{d}{2}\rfloor$.

Since $H_0(n,d)$ is the maximal number of projective vertices obtained from any set of $n$ points in $\mathbb{R}^d$ then, by the UBT, the number of $k$-faces of a projective polytope on $H_0(n,d)$ vertices is bounded by the number of $k$-faces of $C_d(H_0(n,d))$. We thus have  have

\begin{equation}\label{eq:upperbound}
H_{d-1}(n,d) \leq f_{d-1}(C_{d}(H_0(n,d))) \text{ for all } n\ge 1.
\end{equation}

Analogously,  the Lower Bound Theorem \cite{BARNETTE1973, BILLERA1981} states that for all $1\leq k \le d-1$, $$ f_k(P_d(n))\le f_k(P)$$  among all simplicial (convex) polytopes $P \subset \R^d$ with $n$ vertices,  where $P_d(n)$ is a $d$-dimensional {\em stacked polytope} with $n$ vertices, defined as a polytope formed from a simplex by repeatedly gluing another simplex onto one of its facets.

For $d\ge 2$ and $0\le k\le d-1$, the number of $k$-faces of $P_d(n)$ with $n$ vertices is
\begin{eqnarray}\label{form-satcked}
f_{k}(P_{d}(n))=\left\{\begin{array}{ll}
 \binom{d}{k}n - \binom{d+1}{k+1}k &  \text{ if }  0\leq k \leq d-2,\\
 (d-1)n - (d+1)(d-2) &   \text{ if }  k=d-1.\\
 \end{array}\right.
\end{eqnarray}

%In this instance we have the following natural bound for our function:
%\begin{equation}\label{eqn:lowerbound2}
%f_{d-1}(P_d(n))\le H_{d-1}(n,d).
%\end{equation}
As above, we may deduce that
\begin{equation}\label{upperboundcyclicp}
f_{d-1}(P_d(H_0(n,d)))\leq H_{d-1}(n,d).
\end{equation}

\begin{proposition}\label{easybounds} Let $d\ge 2, n\ge 1$ be integers. Then,
$$H_0(n,d)\left\{\begin{array}{ll}
=n & \text{ if } n\le 2d+1,\\
<n & \text{ if } n \ge 2d+\lceil\frac{d+1}{2}\rceil.\\
%=n & \text{ if } d \text{ is even and } n\le d+3,\\
%=n & \text{ if } d \text{ is odd and } n\le d+4.\\
\end{array}\right.$$
\end{proposition}

\begin{proof} Let $n \leq 2d+1$. By the lower bound of $\nu(d)$ given in  \eqref{eq:boundsmcmullen}, it follows that any set of points of cardinality $n$ can be mapped to the vertices of a convex polytope by a permissible projective transformations, and thus $H_{0}(n,d)=n$. If $n \ge 2d + \lceil \frac{d+1}{2} \rceil$ then by the upper bound of $\nu(d)$ given in Equation \eqref{eq:boundsmcmullen},  there   {exists a set of $n$ points that} cannot be mapped to the vertices of a convex polytope by any permissible projective transformation, and thus $H_0(n,d)\le n-1$.

%If $d\ge 2$ is even (resp. is odd) and $n \leq d+3$ (resp. $n\leq d+4$) then, by taking $k=\lfloor \frac{d}{2}\rfloor$, the lower bound of \eqref{eq:boundsmcmullengen}, all polytopes with less than $d+3$ (resp. $d+4$) vertices are in the projective class of a neighbourly polytope (i.e. $k$--neighbourly for $k=\lfloor \frac{d}{2}\rfloor$). The last two equalities follow by using the same argument as above.
\end{proof}

We have the following easy consequence of Proposition \ref{easybounds} and Inequality \eqref{eq:upperbound}.

\begin{proposition}\label{easybounds1} Let $d\ge 2 ,n\ge 1$ be integers. Then,
$$H_{d-1}(n,d)\left\{\begin{array}{ll}
\le f_{d-1}(C_{d}(n)) & \text{ if } n\le 2d+1,\\
\le f_{d-1}(C_{d}(n-1)) & \text{ if } n \ge 2d+\lceil\frac{d+1}{2}\rceil.\\
%=f_{d-1}(C_d(n)) & \text{ if } d \text{ is even and } n\le d+3,\\
%=f_{d-1}(C_d(n)).& \text{ if } d \text{ is odd and } n\le d+4.\\
\end{array}\right.$$
\end{proposition}

Moreover, if  $n\le 2d+1, d\ge 2$ then, by Propostion \ref{easybounds}, $H_0(n,d)=n$  and so $f_{d-1}(P_d(H_0(n,d)))=f_{d-1}(P_d(n))$ obtaining, by Equation \eqref{upperboundcyclicp}, that
\begin{equation}\label{lowerboundHd-1}
f_{d-1}(P_d(n))\le H_{d-1}(n,d).
\end{equation}

%%%%%%%%%%%%%%%%%%%%%%%%%%%%%%%%%%%%
%%%%%%%%%%%%%%%%%%%%%%%%%%%%%%%%%%%%
\section{Oriented matroid setting}\label{sec:lawrence}

Let us briefly give some basic notions and definitions on oriented matroid theory needed for the rest of the paper. We refer the reader to \cite{OM1999} for background on oriented matroid theory.

\subsection{Oriented matroid preliminaries}
Let $M$ be an oriented matroid on a finite set $E$. The matroid $M$ is \emph{acyclic} if it does not contain positive circuits (otherwise, $M$ is called \emph{cyclic}).  A {\em reorientation} of $M$ on $A\subseteq E$ is performed by changing the signs of the elements in $A$ in all the circuits of $M$. It is easy to check that the new set of signed circuits is also the set of circuits of an oriented matroid, usually denoted by $_{-A}M$. A reorientation is {\em acyclic} if $_{-A}M$ is acyclic. An element $e \in E$ of an acyclic oriented matroid is \emph{interior} if there exists a signed circuit $C=(C^+,C^-)$ with $C^-=\{e\}$.

Cordovil and Da Silva \cite{CORDOVIL1985157} proved that a permissible projective transformation on a set $n$ points in $\R^d$ corresponds to an acyclic reorientation of its oriented matroid of affine dependencies $M$ of rank $r=d+1$ and that the converse also holds.

As a consequence of Cordovil and Da Silva's result it is evident that the natural generalization of $n(t,d)$ in terms of oriented matroids is given by the following function:

\begin{quote}
$\bar n(t,d):=$ the largest integer $m$ such that for any uniform {oriented} matroid $M$ of rank $d+1$ with $m$ elements there is an acyclic reorientation of $M$ with at most $t$ interior elements.
\end{quote}

We notice that $n(t,d)=\bar n(t,d)$ in the case when $M$ is {\em realizable}.

In this section we shall provide examples of uniform oriented matroids with the property that in any of their acyclic reorientations there are at least $t+1$ interior elements. These examples provide upper bounds on $\bar{n}(t,d)$. With this aim, we will briefly outline some facts about {\em Lawrence oriented matroids}. For further details and proofs on this special class of matroids see  \cite{OM1999,RAMIREZALFONSIN2001}.

\subsection{Lawrence oriented matroid}
A {\em Lawrence oriented matroid} (LOM) ${M}$ of rank $r$ on the totally ordered set $E=\{1,\ldots,n\}$, $r\le n$, is a uniform oriented matroid obtained as the union of $r$ uniform oriented matroids ${M}_1,\ldots,{M}_r$ of rank $1$ on $(E,<)$. LOMs can also be defined via the signature of their bases, that is, via their chirotope $\chi$. Indeed, the chirotope $\chi$ corresponds to some LOM, ${M}_A,$ if and only if there exists a matrix $A=(a_{i,j})$, $1\le i\le r$, $1\le j\le n$ with entries from $\{+1,-1\}$ (where
the $i$-th row corresponds to the chirotope of the oriented matroid $\mathcal{M}_i$) such that
\begin{equation*}
\chi(B)=\prod_{i=1}^{r} a_{i,j_i}
\end{equation*}
where $B$ is an ordered $r$-tuple, $j_1\le\cdots\le j_r,$ of elements of $E$.

\begin{remark} The following statements about LOMs hold:
\begin{enumerate}[(a)]
\item Acyclic   {LOMs} are realizable as configurations of points (since they are unions of realizable oriented matroids).
\item   {LOMs} are closed under minors and duality.
\item The LOM corresponding to the reorientation of an element $c \in E,$  $_{\bar c}M_A$ is obtained by reversing the sign of all the coefficients of a column $c$ in $A$.
\end{enumerate}
For a proof of this remark see \cite{OM1999}.
\end{remark}

From now on, we will denote $A=A_{r,n}$ as a matrix with entries $a_{i,j}\in \{+1,-1\}$, $1\le i\le r$, $1\le j\le n$.
Some of the following definitions and lemmas, which highlight the properties of $A$ and facilitate the study of  this type of matroid, were introduced and proved in \cite{RAMIREZALFONSIN2001}.

A \emph{Top Travel}, denoted as $TT,$ in $A$ is a subset of the entries of $A$,
$$\{ [a_{1, 1}, a_{1, 2},\dots ,a_{1, j_{1}}] , [a_{2, j_{1}}, a_{2, j_{1}+1},\dots , a_{2, j_{2}}] , \dots , [a_{s, j_{s-1}}, a_{s, j_{s-1}+1},\dots , a_{s, j_{s}}] \},$$
where $[a_{l, j_{l-1}},\dots ,a_{l, j_{l}}]$ are the entries in line $l$,
with the following characteristics:
\begin{enumerate}
\item$ a_{i, j_{i-1}} \times a_{i, j}= 1, \quad \forall \quad\  j_{i-1} \leq j < j_{i};$
\item$ a_{i, j_{i-1}} \times a_{i, j_{i}}= -1;$  and
\item either
\begin{enumerate}
\item $1\leq s < r$;   then  $j_{s} = n$ or
\item $s=r$ and $ j_{s} \leq n.$
\end{enumerate}
\end{enumerate}

A \emph{Bottom Travel}, denoted as $BT$, in $A$ is a subset of the entries of $A$,
$$\{ [a_{r, n}, a_{r, n-1},\dots ,a_{r, j_{r}}] , [a_{r-1, j_{r}}, a_{r-1, j_{r}-1},\dots , a_{r-1, j_{r-1}}] , \dots , [a_{s, j_{s-1}}, a_{s, j_{s-1}+1},\dots , a_{s, j_{s}}] \},$$ with the following characteristics:
\begin{enumerate}
\item $ a_{i, j_{i+1}} \times a_{i, j}= 1, \quad \forall \quad\  j_{i} < j \leq j_{i+1};$
\item $ a_{i, j_{i+1}} \times a_{i, j_{i}}= -1;$  and
\item either
\begin{enumerate}
\item $1 < s \leq r$;   then  $j_{s} = 1$ or
\item $s=1$ and $ 1 \leq j_{s}.$
\end{enumerate}
\end{enumerate}

Every matrix $A$ has exclusively one $TT$ and one $BT$ and they carry surprising information about $M_{A}$.

\begin{remark} Let $A$ be a $r \times n$-matrix, then the following statements are equivalent:
\begin{enumerate}[(a)]
\item$M_{A}$ is \emph{cyclic};
\item $TT$ ends at $ a_{r,s}$ for some $ 1 \leq s < n$; and
\item $BT$ ends at $ a_{1, s'}$ for some $ 1< s' \leq n$.
\end{enumerate}
For a proof of this remark see \cite{RAMIREZALFONSIN2001}.
\end{remark}

Let $a_{i,k-1},a_{i,k},a_{i,k+1}\in TT$ we say that $TT$ and $BT$ are \emph{parallel} at column $k$ if either $a_{i,k-1},a_{i,k},a_{i,k+1}\in BT$ or $a_{i+1,k-1},a_{i+1,k},a_{i+1,k+1}\in BT$, with $2\le k\le n-1$, $1\le i\le r$.

\begin{remark} \label{rem:interior}  \cite{RAMIREZALFONSIN2001}
Let $A$ be a $r \times n$-matrix then $k$ is an interior element of $M_A$ if and only if
\begin{enumerate}[(a)]
\item$BT=(a_{r,n},\ldots,a_{1,2},a_{1,1})$ for $k=1$,
\item $TT=(a_{1,1},\ldots,a_{r,n-1},a_{r,n})$ for $k=n$,
\item $TT$ and $BT$ are parallel at $k$ for $2\le k\le n-1$.
\end{enumerate}
\end{remark}

Remark \ref{rem:interior} implies that we can identify acyclic reorientations and interior elements of an oriented matroid $M_{A}$ by studying the behaviour of the $TT$ and $BT$ in the re-orientations of $A$.

\begin{example} Let $M_A$ be the LOM associated to the matrix $A$ described in Figure \ref{ejemplo}. $M_A$ is acyclic, and $4$, $5$ and $6$ are interior elements.
\begin{figure}[htb]
\begin{center}
 \includegraphics[width=.4\textwidth]{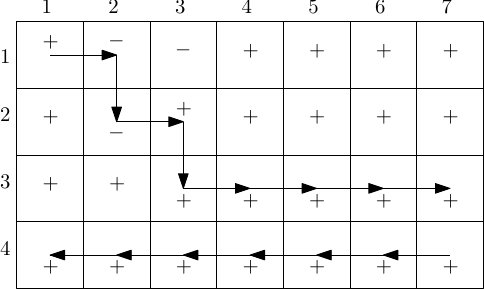}
\caption{Top and Bottom travels in matrix $A$.} \label{ejemplo}
\end{center}
\end{figure}
\end{example}

Furthermore, all possible re-orientations of the matroid can be identified with yet another simple object;

A \emph{Plain Travel}  in $A$, denoted as $PT$, is a subset of the entries of $A$ which satisfies:
$$PT=\{ [a_{1, 1}, a_{1, 2},\dots ,a_{1, j_{1}}] , [a_{2, j_{1}}, a_{2, j_{1}+1},\dots , a_{2, j_{2}}] , \dots , [a_{s, j_{s-1}}, a_{s, j_{s-1}+1},\dots , a_{s, j_{s}}] \} $$  with \( 2\leq j_{i-1}  {<} j_{i}  \leq n$ for all $1 \leq i \leq r ,\;\; 1< s \leq r$   and  $j_{s} = n$.

 \begin{remark} There is a bijection between the set of all plain travels of $A$ and the set of all acyclic reorientations of $M_{A}$, it is defined by associating to each $PT$ the set of column indices of $A$ that should be reoriented in order to transform $A$ into a new matrix $\mathcal A$ whose $TT$ is identical to $PT$. For a proof of this remark see \cite{RAMIREZALFONSIN2001}.
\end{remark}

The \emph{chessboard} $B[A]$ of $A$ is another useful object that can be constructed from its entries.  It is defined as a black and white board of size $(r-1) \times (n-1)$, such that the square $s(i,j)$ has its upper left hand corner at the intersection of row $i$ and column $j$; a square $s(i,j)$, with $1 \leq i \leq r-1$ and $ 1 \leq j \leq n-1$ will be said to be {\em black} if the product of the entries $a_{i,j}, a_{i,j+1}, a_{i+1,j}, a_{i+1,j+1}$ is $-1$, and {\em white} otherwise. See Figure \ref{ejemplo2} for an example.

\begin{figure}[htb]
\begin{center}
 \includegraphics[width=.4\textwidth]{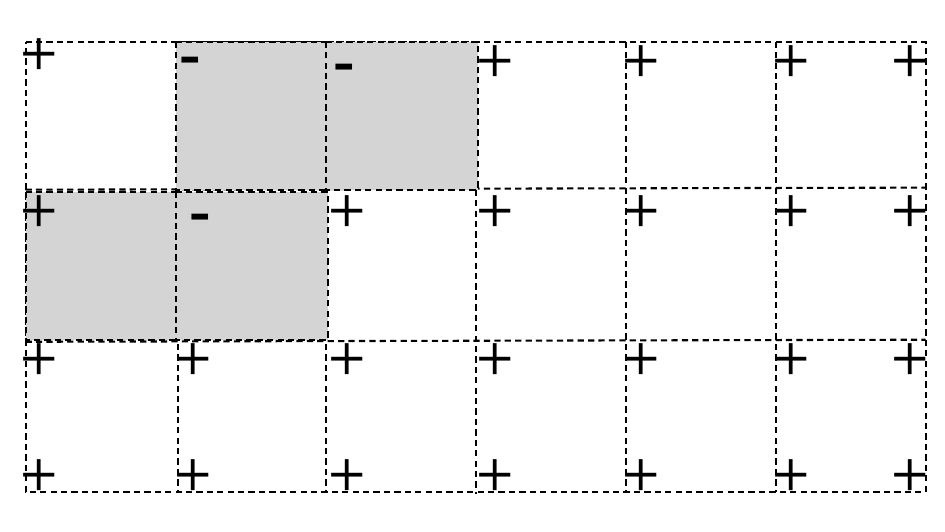}
\caption{The chessboard $B[A]$ of the matrix $A$ described in Figure \ref{ejemplo}.} \label{ejemplo2}
\end{center}
\end{figure}

\begin{remark} \label{prop:chessboard}  \cite{RAMIREZALFONSIN2001} The following statements establish a link between chessboards, reorientations of a LOM, and $BT$ and $TT$:
\begin{enumerate}[(a)]
\item $B[A]$ is invariant under reorientations of $M_A$.
\item If in a pair of consecutive columns there is one black square between the top travel $TT$ and the bottom travel $BT$, they follow symmetrically opposite paths through the entries of the matrix; in other words, if $TT$ makes a single horizontal movement from $a_{i,j}$ to $a_{i,j+1}$ and continues its movement forward in the same row (i.e $a_{i,j} = a_{i,j+1}$) , then $BT$ goes from  $a_{i+h,j+1}$ to $a_{i+h,j}$  and moves vertically to  $a_{i+h-1, j}$ (i.e. $a_{i+h,j+1}\neq a_{i+h,j}$), with $h\geq{1}$, and vice versa.
\end{enumerate}
\end{remark}

%%%%%%%%%%%%%%%%%%%%%%%%%%%%%%%%%%%%
%%%%%%%%%%%%%%%%%%%%%%%%%%%%%%%%%%%%
\section{Upper bounds for $\bar{n}(d,t)$}\label{cotassupmatroides}

\subsection{Small dimension $d$} We first show that $\bar{n}(t,d)\le 2d+1+t$ for $d=2,3$ and every $t\ge0$. The following remark will be very useful throughout this section.

 \begin{remark} \label{dim1}
Any acyclic reorientation of a rank $2$ oriented matroid on $n$ elements  has %exactly
 $n-2$ interior elements.
\end{remark}

Given a matrix $A=A_{n,r}$, let $A^+_{i,j}$ be  the sub-matrix of $A$  that results after removing rows $i+1,\ldots, r$ and columns $j+1,\ldots, n$. Similarly, let $A^-_{i,j}$ be the sub-matrix of $A$ resulting after the removal of rows $1,\ldots, i-1$ and columns $1,\ldots,j-1$.

 \begin{theorem}\label{dimension2}
$\bar{n}(t,2)<t+6$ for every integer $t\ge 0$.
\end{theorem}\begin{proof}
Let $A=A_{3,t+6}$  be such that the corresponding chessboard $B[A]$ has exactly one black square for each column and   let  $PT$ be any plain travel in $A$. We shall prove that the corresponding $\mathcal{A}$ in which $PT$ is the Top Travel has at least $t+1$ interior elements.
 Let $j$ be the smallest number such that column $j$ is not an interior element in $\mathcal{A}$ and there are not vertical movements in column $j$ of $PT$ neither of $BT$. %(the corresponding Bottom Travel  in $\mathcal{A}$)
  If $j$ does not exists, as $PT$ and $BT$ can make at most $2$ vertical movements each, then $\mathcal{A}$ would have at least $t+2$ interior elements. Hence, we may suppose that $j$ exists.  By the rules of  {Proposition} \ref{prop:chessboard}, $PT$ and $BT$  make a vertical movement in column $j+1$ and $j-1$, respectively.
   Notice by the definition of $j$ that $PT$ and $BT$ arrives in column $j$ at row $1$ and $3$, respectively, otherwise $j$ would be an interior element. Then,  each interior element of $A^+_{2,j-1}$  and $A^-_{2,t+6-j}$ is an interior element of $\mathcal{A}$. Therefore,   $A^+_{2,j-1}$  has  $j-3$ interior elements and  $A^-_{2,t+6-j}$ has $t+6-j-2$ interior elements  by  Remark \ref{dim1}, concluding the proof.
 \end{proof}

Given a matrix $A=A_{r,n}$, we say that a chess board $B[A]$ has the \emph{sequence} $(x_1,x_2,\ldots,x_{r-1})$ if the square $s(i,j)$ is black if and only if $\sum\limits_{k=0}^{i-1}x_{k}+1\le j\le \sum\limits_{k=0}^{i}x_{k}$ with $1 \leq i \leq r-1$ and $1 \leq j \leq n-1$, where we define $x_0=0$.

\begin{example} Figure \ref{figurehypothesis} illustrates a chessboard with sequence $(2,3,2,3)$.
\end{example}

 \begin{theorem}\label{dimension3}
$\bar{n}(t,3)<t+8$ for every integer $t\ge 0$.
\end{theorem}\begin{proof}
 Let $A=A_{4,t+8}$  be such that the corresponding chessboard $B[A]$ has a sequence $(2,t+3,2)$ and  let  $PT$ be any plane travel in $A$. We prove that the corresponding $\mathcal{A}$ in which $PT$ is the Top Travel has at least $t+1$ interior elements. Let $j$ be the smallest number such that column $j$ is not an interior element in $\mathcal{A}$ and there are not vertical movements in column $j$ of $PT$ neither of $BT$.
  If $j$ does not  {exist}, as $PT$ and $BT$ can make at most $3$ vertical movements each, then $\mathcal{A}$ would have at least $t+2$ interior elements. Hence, we may suppose that $j$ exists.  By the rules of  {Proposition} \ref{prop:chessboard}, $PT$ and $BT$  make a vertical movement in column $j+1$ and $j-1$, respectively.

By the definition of  $j$, $TT$ arrives in column $j$ at row $1$ or $BT$ arrives in column $j$ at row $4$, otherwise $j$ would be an interior element. Suppose without loss of generality that $TT$ arrives in column $j$ at row $1$. Then, $BT$ arrives in column $j$ at row $2$ or $3$ (see Figure \ref{dim3}).
   \begin{figure}[htb]
\begin{center}
 \includegraphics[width=.2\textwidth]{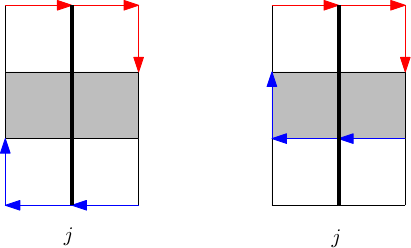}
\caption{$BT$ in blue and $TT$ in red.}
 \label{dim3}\end{center}
\end{figure}
Notice that each interior element of $A^+_{3,j-1}$   and $A^-_{2,t+8-j}$ is an interior element of $\mathcal{A}$. If $BT$ arrives in column $j$ at row $3$, by the rules of  {Proposition} \ref{prop:chessboard}, $A^+_{3,j-1}$  has  $j-4$ interior elements and $A^-_{2,t+8-j}$ has $t+5-j$ interior elements, concluding the proof in this case.  If $BT$ arrives in column $j$ at row $2$, by the rules of  {Proposition} \ref{prop:chessboard} and by Remark \ref{dim1}, one can check that $A^+_{3,j-1}$  has  $j-3$ interior elements and $A^-_{2,t+8-j}$ has at least $t+4-j$ interior elements, concluding the proof.
 \end{proof}

%%%%%%%%%%%%%%%%%%%%%%%%%%%%%%%%%%%%
%%%%%%%%%%%%%%%%%%%%%%%%%%%%%%%%%%%%
\subsection{High dimension $d$}
In what follows, we will consider different matrices $A=A_{r,h(r)}$ where $h(r)$ is a strictly increasing function. If $B[A]$ has the sequence $(x_1,x_2,\ldots,x_{r-1})$, we will consider functions, $h(r),$ where $h(1)=1$, $\sum\limits_{k=1}^{m-1}x_{k}+1\le h(m)\le \sum\limits_{k=1}^{m}x_{k}+1$ if  $2\le m\le r-1$ and $\sum\limits_{k=1}^{r-1}x_{k}+1\le h(r)$.

For every $1\le m\le r-1$, we will say that the element $a_{m,h(m)}$  is  the \emph{$m$--th corner} of $A$.

\begin{example} Figure \ref{figurehypothesis} illustrates a chessboard with $h(r)=2(r-1)+\lceil\frac{r}{2}\rceil$.
\end{example}

The following lemmas will be helpful ingredients for our proposes.

 \begin{lemma}\label{lemmageneral1}
 Let $A=A_{r,h(r)}$ be a matrix with $r\ge 3$ such that $B[A]$ has the sequence $(x_1,x_2,\ldots,x_{r-1})$, with $x_i\ge1$, $1\le  i\le r-1$. Suppose that $TT$ always passes strictly above all corners after  the $1$--st corner, then

\begin{itemize}
 \item [(i)] $a_{i,h(2)}\in BT$ for some $i\le \max\{2,2r-h(r-1)+h(2)-3\}$,
\item [(ii)] $a_{i,h(2)}\in BT$ for some $i\le \max\{2,2r-h(r)+h(2)-1\}$ if $TT$ and $BT$ do not share steps from columns $h(2)$ to $h(r)$.
 \end{itemize}
 \end{lemma}
\begin{proof}
We will proceed by induction  on $r\ge3$.  Using the rules of  {Proposition} \ref{prop:chessboard},  one can check that the lemma holds for $r=3$.
Suppose that the result holds for $r-1$ and we show it for $r\ge4$. If $BT$ arrives at the $m$--th corner for $3\le m\le r-1$, the result follows by the induction hypothesis. Moreover, if $BT$ arrives at the $2$--th corner, the result follows.
Similarly, the result follows if $BT$ arrives at $a_{i,h(m)}$ for $2\le m\le r-1$ and $i\le m$.  Then, we may suppose that $BT$ always passes strictly below the $m$--th corner for every $2\le m\le r-1$.

Notice that $TT$ makes exactly $h(r)-h(2)$ horizontal movements and at most $r-1$ vertical movements, from right to left, to arrive at  $a_{1,h(2)}$. We know, by the rules of construction of $TT$ that for each vertical movement we must also count one horizontal movement. So, $TT$ makes at least $h(r)-h(2)-(r-1)$ single horizontal movements, from right to left, until   $a_{1,h(2)}$ is attained.

On the other hand, as $TT$ always passes strictly above all corners after  the $1$--st corner and $BT$ always passes strictly below the $m$--th corner for every $2\le m\le r-1$,  then $TT$ and $BT$  do not share steps from columns $h(2)$ to $h(r-1)$. However, $TT$ and $BT$ could share at most $h(r)-h(r-1)-2$ steps from columns $h(r-1)+1$ to $h(r)$ (see Figure \ref{figurehypothesis}).
\begin{figure}[htb]
\begin{center}
 \includegraphics[width=.35\textwidth]{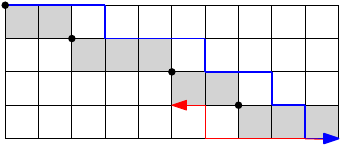}
 \caption{The chessboard $B[A]$ with sequence $(2,3,2,3)$ and $h(r)=2(r-1)+\lceil\frac{r}{2}\rceil$. The points represent the corners of $A=A_{5,h(5)}$. We observe that $TT$ and $BT$ share the final step.} \label{figurehypothesis}
 \end{center}
\end{figure}

Therefore, by the rules of  {Proposition} \ref{prop:chessboard},  for each single horizontal movement that $TT$ does not share with $BT$, $BT$ makes a vertical movement. So, $BT$ makes at least  $h(r)-h(2)-(r-1)-(h(r)-h(r-1)-2)=h(r-1)-h(2)-r+3$ vertical movements, from right to left, until   column $h(2)$ is attained. Hence,
 $BT$ arrives at $a_{i,h(2)}$ for some $i\le \max\{2,r-(h(r-1)-h(2)-r+3)\}$  concluding the first part of the proof.  If  $TT$ and $BT$ do not share steps from columns $h(2)$ to $h(r)$, then $BT$ makes at least  $h(r)-h(2)-(r-1)$ vertical movements, from right to left, until   column $h(2)$  is attained, concluding that
 $BT$ arrives at $a_{i,h(2)}$ for some $i\le \max\{2,r-(h(r)-h(2)-(r-1))\}$.
  \end{proof}

\begin{lemma}\label{lemmageneral2}Let  $A=A_{r,h(r)}$ be a matrix and suppose that  $TT$ always passes strictly above all corners after  the $1$--st corner. Then the following holds:
\begin{itemize}
  \item [(i)] If $B[A]$ has a sequence $(x_1,x_2,\ldots,x_{r-1})$, $x_i\ge2$  for odd $i$, $x_j\ge3$ for even $j$ and $h(m)=\sum\limits_{k=0}^{m-1}x_{k}+1$ for every $1\le m\le r$,   then $a_{1,1},a_{1,2}\in BT$ when $r\ge4$.  When $r=3$,  $a_{1,1}, a_{1,2}\in BT$, or column $h(r)$ is an interior element and $a_{2,1},a_{1,1}\in BT$.

  \item [(ii)]   If $B[A]$ has a sequence $(x_1,x_2,\ldots,x_{r-1})$, $x_i\ge3$  for odd $i$, $x_j\ge2$ for even $j$ and $h(m)=\sum\limits_{k=0}^{m-1}x_{k}+1$ for every $1\le m\le r$, then $a_{1,1},a_{1,2}\in BT$ when $r\ge3$.

  \item [(iii)] If $B[A]$ has a sequence $(2,t+3,2,t+1,t+1,\ldots,t+1)$ for some $t\ge2$,  $h(2)=t+3$, $h(3)=t+6$ and $h(m)=(t+1)(m-3)+7$ for $4\le m\le r$ for every  $4\le m\le r$, then $a_{i,t+3}\in BT$ for some $i\le 2$ when $r\ge4$ (see Figure \ref{cotageneral2}).

\item [(iv)]  If $B[A]$ has a sequence $(t+1,\ldots,t+1,2,t+3,2)$ for some $t\ge2$,  $h(2)=t+3$, %una unidad adelante de lo habitual...
 $h(r)=(t+1)(r-3)+7$ and $TT$ and $BT$ do not share steps from columns $h(2)$ to $h(r)$, then $a_{i,t+3}\in BT$ for some $i\le 2$ when $r\ge5$.

  \item [(v)] If  $B[A]$ has a sequence $(2,4,2,3,2,3,\ldots)$ and $h(m)=2(m-1)+\lceil\frac{m}{2}\rceil+1$ for every $2\le m\le r$,  then $a_{i,4}\in BT$ for some $i\le 2$ when $r\ge4$.

  \item [(vi)] If $r\ge6$  is even,
       $B[A]$ has a sequence $(2,3,2,3,\ldots,2,4,2)$, $h(2)=4$, $h(r)=2(r-1)+\lceil\frac{r}{2}\rceil+1$  and $TT$ and $BT$ do not share steps from columns $h(2)$ to $h(r)$, then $a_{i,4}\in BT$ for some $i\le 2$.
\end{itemize}
\end{lemma}
\begin{proof}
We prove (i) and (ii) for $B[A]$ with sequences  $(2,3,2,3\ldots)$ and  $(3,2,3,2\ldots)$, respectively, since the general case holds as a consequence.

  (i) By the sequence of $B[A]$, we observe that $h(m)=2(m-1)+\lceil\frac{m}{2}\rceil$ for $1\le m\le r$. Then, as $2r-h(r-1)+h(2)-3=4-\lceil\frac{r-1}{2}\rceil\le 2$ when $r\ge4$, we obtain by Lemma  \ref{lemmageneral1} (i) that  $a_{i,h(2)}\in BT$ for some $i\le2$. Since $a_{1,i}\in TT$ for $i\le4$ and $h(2)=3$, we conclude by the rules of  {Proposition} \ref{prop:chessboard} that $a_{1,1},a_{1,2}\in BT$.
  If $r=3$, one can check that $a_{1,1}, a_{1,2}\in BT$, or  column $6$ is an interior element and $a_{1,1}, a_{2,1}\in BT$ (see Figure \ref{r3}).
\begin{figure}[htb]
\begin{center}
 \includegraphics[width=.2\textwidth]{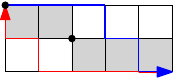}
 \caption{Case $r=3$ when column $6$ is an interior element and $a_{1,1}, a_{2,1}\in BT$.  The points represent the corners.} \label{r3}
 \end{center}
\end{figure}

 (ii) By the sequence of $B[A]$, we observe that  $h(m)=2(m-1)+\lceil\frac{m+1}{2}\rceil$ for $1\le m\le r$. Then, as $2r-h(r-1)+h(2)-3=5-\lceil\frac{r}{2}\rceil\le 3$, we obtain by Lemma  \ref{lemmageneral1} (i) that $a_{i,h(2)}\in BT$ for some  $i\le3$. Since $a_{1,i}\in TT$ for $i\le5$  and $h(2)=4$,  we conclude by the rules of  {Proposition} \ref{prop:chessboard} that $a_{1,1},a_{1,2}\in BT$.

 (iii)   As  $2r-h(r-1)+h(2)-3=2$ when $r=4$ and $2r-h(r-1)+h(2)-3=2r-(t+1)(r-4)-7+t\le 7-r\le2$ when  $r\ge5$ and $t\ge2$, we obtain by Lemma  \ref{lemmageneral1} (i) that  $a_{i,h(2)}\in BT$ for some $i\le 2$.

(iv) As $2r-h(r)+h(2)-1=2r-(t+1)(r-3)+t-5\le 6-r\le1$ when  $r\ge5$ and $t\ge2$, we obtain by Lemma  \ref{lemmageneral1} (ii) that  $a_{i,h(2)}\in BT$ for some $i\le 2$.

 (v)  As $2r-h(r-1)+h(2)-3=4-\lceil\frac{r-1}{2}\rceil\le2$ when $r\ge4$, we obtain by Lemma   \ref{lemmageneral1} (i) that  $a_{i,h(2)}\in BT$  for some $i\le 2$.

 (vi)  As $2r-h(r)+h(2)-1=2r-2(r-1)-\frac{r}{2}+2=4-\frac{r}{2}\le1$ when $r\ge6$ is even, we obtain by Lemma   \ref{lemmageneral1} (ii) that  $a_{i,h(2)}\in BT$  for some $i\le 2$.
\end{proof}
%%%%%%%%%%%%%%%%%%%%%%%%%%

From now on,  denote as $A^+_m$ and $A^-_m$ the matrices $A^+_{m,h(m)}$ and $A^-_{m,h(m)}$, respectively (see Figure \ref{cotageneral2}).

We are now ready to tackle the case when $d\ge4$ and $t\ge2$.

 \begin{theorem}\label{cotageneral}
 $\bar{n}(t,d)<2d+(t-1)(d-2)+3$  for integers $d\ge4$ and $t\ge 2$.
    \end{theorem}
\begin{proof}
Let $A=A_{r,h(r)}$  be a matrix  where $h(r)$ is defined as
$h(2)=t+3$, $h(3)=t+6$, $h(m)=(t+1)(m-3)+7$ for $4\le m\le r$ and  $B[A]$ with sequence $(2,t+3,2,t+1,t+1,\ldots,t+1)$ for $t\ge2$ (see Figure \ref{cotageneral2}).
 %%%%%%%%%%%%%%FIGURE GENERAL CHESSBOARD %%%%%%%%5
 \begin{figure}[htb]
\begin{center}
 \includegraphics[width=.6\textwidth]{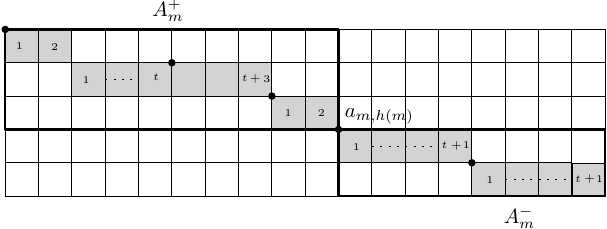}
\caption{A matrix $A=A_{6,h(6)}=A_{6,19}$ and the sub-matrices $A^+_4$ and $A^-_4$. The chessboard  $B[A]$ has sequence $(2,t+3,2,t+1,t+1)$ for $t=3$.   The points represent the corners of $A$ associated to the function $h(r)$ of Theorem \ref{cotageneral}.} \label{cotageneral2}\end{center}
\end{figure}
%%%%%%%%%%%%%%%%%%%%%%%%%%%%%%
 We shall show  by induction on $r$ that for every $r\ge2$ and for any  plain travel $PT$ in $A$, the corresponding $\mathcal{A}$ in which $PT$ is the Top Travel  has at least $t+1$ interior elements. In particular, as $h(r)=(t+1)(d-2)+7=2d+(t-1)(d-2)+3$ for $r\ge5$, we will prove the theorem for $d\ge4$ and $t\ge2$.

We observe that $\mathcal{A}$  has  $t+1$ interior elements when $r=2$ since $\mathcal{A}$ is a $2\times (t+3)$ matrix (Remark \ref{dim1}).
 For $r=3$ and $4$, the result follows by Theorems \ref{dimension2} and \ref{dimension3}, respectively, since the  chessboards  considered in these theorems coincide with  $B[A]$. Thus, assume that the theorem holds for $r-1$ and we show it for $r\ge 5$. Suppose that  the $m$--st corner is the last corner that $PT$ meets in $A$, for some  $1\le m\le r-1$. If  $PT$ always passes strictly below the $i$--st corner for  $i>m$, then there would be at least $t+2$ interior elements in $\mathcal{A}$ (from columns $h(r-1)$ to $h(r)$). Hence,  we may suppose that $PT$ always passes strictly above the $i$--st corner for  $i>m$. We have the following cases.

\emph{Case $m\le r-3.$}
First suppose that $m=1$. Then $a_{i,t+3}\in BT$ for some $i\le 2$  by Lemma \ref{lemmageneral2} (iii), concluding by the rules of  {Proposition}  \ref{prop:chessboard} that there are at least $t+1$ interior elements in $\mathcal{A}$. Now suppose that $2\le m\le r-3.$
As $\mathcal{A}^-_m$  has at least $4$ rows, applying  Lemma \ref{lemmageneral2} (i)  (when $m=3$) and   Lemma \ref{lemmageneral2} (ii) (when $m\neq 3$) on sub-matrix $\mathcal{A}^-_m$, we obtain that  $a_{m,h(m)},a_{m,h(m)+1}\in BT$. Thus, the theorem holds  by induction  hypothesis on $\mathcal{A}^+_m$ since $BT$  restricted in  $\mathcal{A}^+_m$ is also the Bottom travel of $\mathcal{A}^+_m$ and each interior element of $\mathcal{A}^+_m$ is an interior element of $\mathcal{A}$.

\emph{Case $m=r-2.$}
As $\mathcal{A}^-_m$  has $3$ rows, $a_{m,h(m)},a_{m,h(m)+1}\in BT$, or  column $h(r)$ is an interior element and $a_{m+1,h(m)}, a_{m,h(m)}\in BT$ by Lemma \ref{lemmageneral2} (i). If $a_{m,h(m)},a_{m,h(m)+1}\in BT$, the theorem holds by induction hypothesis on $\mathcal{A}^+_m$. If column $h(r)$ is an interior element and $a_{m+1,h(m)}, a_{m,h(m)}\in BT$, notice that each interior element of $\mathcal{A}^+_m$ is an interior element of $\mathcal{A}$, except for column $h(m)$. Thus, $\mathcal{A}$ has at least $t$ interior elements from columns $1$ to $h(m)$ by induction hypothesis on $\mathcal{A}^+_m$ and one interior element in  column $h(r)$.

\emph{Case $m=r-1.$}
First suppose that $BT$ arrives at the $k$-th corner for some  $2\le k\le r-1$. As each interior element of $\mathcal{A}^+_k$ and $\mathcal{A}^-_m$ is an interior element of $\mathcal{A}$, except for (maybe) columns $h(k)$ and $h(m)$,  $\mathcal{A}$ has at least $t$ interior elements from columns $1$ to $h(k)$ by induction hypothesis on $\mathcal{A}^+_k$ and at least $t-1\ge 1$ interior elements from columns $h(m)$ to $h(r)$ by Remark \ref{dim1} (since $\mathcal{A}^-_m$ is a $2\times (t+2)$ matrix), concluding the proof in this case. Similarly, the proof holds if $BT$ arrives at $a_{i,h(k)}$ for $2\le k\le r-2$ and $i\le k$.
        Now suppose that $BT$ passes always below the $i$--st corner, for every $i\ge2$. In particular, as $BT$ does not arrives at the $m$-th corner, every interior element of $\mathcal{A}^-_m$  is an interior element of $\mathcal{A}$, concluding by Remark \ref{dim1} that $\mathcal{A}$ has $t$ interior elements from columns $h(m)+1$ to $h(r)$. So, we may suppose  that $TT$ and $BT$ do not share steps from columns $1$ to $h(2)$, otherwise the theorem holds.  Also, if $a_{m,h(m)-1}\in PT$, then column $h(m)$ is an interior element and the theorem holds. So, we may suppose that $a_{m,h(m)-1}\not\in PT$. Hence, $a_{r,h(m)-2}\in BT$ by the rules of  {Proposition} \ref{prop:chessboard}.
     Let $\mathcal{A}^{'}$ be the matrix obtained by turning the matrix $\mathcal{A}$ upside down. We observe that $BT$ and $PT$ are the Top and Bottom Travels of $\mathcal{A}^{'}$, respectively. Let  {us} define the $i$--st corners of $\mathcal{A}^{'}$ as $a_{r-i+1,h(r-i+1)}$ for $i\neq 2$ and define the
      $2$--st corner of $\mathcal{A}^{'}$ as $a_{m,h(m)-1}$.

     Notice that $BT$ always passes strictly above all corners of  $\mathcal{A}^{'}$ after the $1$--st corner of $\mathcal{A}^{'}$. Moreover, $B[A']$ has the same sequence and the same $2$--st corner as that considered in Lemma \ref{lemmageneral2} (iv). Hence, as $TT$ and $BT$ do not share steps from columns $h(1)$ to $h(m)-1$,
     we know  by Lemma \ref{lemmageneral2} (iv) that $a_{i,h(m)-1}\in PT$ for some $i\ge r-1$, but this is a contradiction since  we had assumed that $a_{m,h(m)-1}\not\in PT$ and clearly $a_{r,h(m)-1}\not\in PT$, concluding the proof.
\end{proof}
Figure \ref{contraexample1} (a) shows that the chessboard considered in Theorem \ref{cotageneral}  {cannot} be used to prove  $\bar{n}(d,1)<2d+(t-1)(d-2)+3$  for $d=4$ and $t=1$. In fact, this example can be generalized in order to show that this chessboard  {cannot} be used to prove   $\bar{n}(d,t)<2d+3$ for $d\ge4$.

\begin{theorem}\label{tpequeno}
 $\bar{n}(1,d)<2d+\lceil\frac{d+1}{2}\rceil+1$ for any integer $d\ge4$.
 \end{theorem}
\begin{proof}
Let $A=A_{r,h(r)}$  be a matrix  where  $h(r)$ is defined as  $h(m)=2(m-1)+\lceil\frac{m}{2}\rceil+1$ for every $2\le m\le r$ and $B[A]$ has sequence $(2,4,2,3,2,3,\ldots)$. We shall show  by induction on $r$ that for every $r\ge2$ and for any  plain travel $PT$ in $A$, the corresponding $\mathcal{A}$ in which $PT$ is the Top Travel  has at least $2$ interior elements. In particular, as $h(r)=2d+\lceil\frac{d+1}{2}\rceil+1$ for $r\ge5$, we will prove the theorem for $d\ge4$.

We observe that $\mathcal{A}$  has at least  $2$ interior elements when $r=2$ since $\mathcal{A}$ is a $2\times4$ matrix (Remark \ref{dim1}). For $r=3$ and $4$, the result follows by Theorems \ref{dimension2} and \ref{dimension3}, respectively (applying them for $t=1$), since the  chessboards  considered in these theorems coincide with  $B[A]$. Thus, assume the theorem holds for $r-1$ and we show it for $r\ge 5$. Suppose that  the $m$--st corner is the last corner that $PT$ meets in $A$, for some  $1\le m\le r-1$. If  $PT$ always passes strictly below the $i$--st corner for  $i>m$, then there would be at least $3$ interior elements in $\mathcal{A}$ (from columns $h(r-1)$ to $h(r)$). Hence,  we may suppose that $PT$ always passes strictly above the $i$--st corner for  $i>m$. If $m=1$, then $a_{i,4}\in BT$ for some $i\le 2$ by Lemma \ref{lemmageneral2} (v), concluding by the rules of  {Proposition}  \ref{prop:chessboard} that there are at least $2$ interior elements in $\mathcal{A}$. We omit the proof of the cases $2\le m\le r-2$, since they are analogous to those in the proof of Theorem \ref{cotageneral}. Now, consider the case $m=r-1$.

If $r$ is odd,  $B[A]$ has a sequence $(2,4,2,3,2,3,\ldots,2,3)$ and one can verify that there are at least $2$ interior elements in $\mathcal{A}$ or  $a_{m,h(m)},a_{m,h(m)+1} \in BT$. In both cases the theorem holds. Now suppose that $r$ is even and then $B[A]$ has a sequence $(2,4,2,3,2,3,\ldots,3,2)$.
If $a_{m,h(m)},a_{m,h(m)+1} \in BT$,  the result follows by induction hypothesis on $\mathcal{A}^+_m$. Then, suppose from now that the above does not hold. Hence, each interior element of $\mathcal{A}^-_m$   is an interior element of $\mathcal{A}$ concluding  by Remark \ref{dim1} that  $\mathcal{A}$ has one interior element from columns $h(m)+1$ to $h(r)$ (since $\mathcal{A}^-_m$ is a $2\times 3$ matrix). First suppose that $BT$ arrives at the $k$-th corner for some  $2\le k\le r-1$. As each interior element of $\mathcal{A}^+_k$  is an interior element of $\mathcal{A}$, except for (maybe) column $h(k)$, $\mathcal{A}$ has at least one interior element from columns $1$ to $h(k)$ by induction hypothesis on $\mathcal{A}^+_k$ and since  $\mathcal{A}$ has one interior element from columns $h(m)+1$ to $h(r)$, the theorem holds in this case. Similarly, the proof holds if $BT$ arrives at $a_{i,h(k)}$ for $2\le k\le r-2$ and $i\le k$. Now suppose that $BT$ passes always below the $i$--st corner, for every $i\ge2$. If $TT$ and $BT$ share steps from columns $1$ to $h(2)$, then $\mathcal{A}$  has at least one interior element from columns $1$ to $h(2)$ and since $\mathcal{A}$ has one interior element from columns $h(m)+1$ to $h(r)$, the theorem holds. Then, we may suppose that $TT$ and $BT$ does not share steps from columns $1$ to $h(2)$.  Also,  if $a_{m,h(m)-1}\in PT$, then column $h(m)$ is an interior element and the theorem holds. So, we may suppose that $a_{m,h(m)-1}\not\in PT$. Hence, $a_{r,h(m)-2}\in BT$ by the rules of  {Proposition}  \ref{prop:chessboard}. Let $\mathcal{A}^{'}$ be the matrix obtained by turning the matrix $\mathcal{A}$ upside down. We observe that $BT$ and $PT$ are the Top and Bottom Travels of $\mathcal{A}^{'}$, respectively.

Let define the $i$--st corners of $\mathcal{A}^{'}$ as $a_{r-i+1,h(r-i+1)}$ for $i\neq 2$ and define the
      $2$--st corner of $\mathcal{A}^{'}$ as $a_{m,h(m)-1}$.    Notice that $BT$ always passes strictly above all corners of  $\mathcal{A}^{'}$ after the $1$--st corner of $\mathcal{A}^{'}$. Moreover, $B[A']$ has the same sequence and the same $2$--st corner as that considered in Lemma \ref{lemmageneral2} (vi). Hence, as $TT$ and $BT$ do not share steps from columns $h(1)$ to $h(m)-1$, we know  by Lemma \ref{lemmageneral2} (vi) that $a_{i,h(m)-1}\in PT$ for some $i\ge r-1$, but this is a contradiction since  we had assumed that $a_{m,h(m)-1}\not\in PT$ and clearly $a_{r,h(m)-1}\not\in PT$, concluding the proof.
\end{proof}

The chessboard considered in Theorem \ref{cotageneral} {cannot} be extended to a chessboard with sequence $(2,t+3,2,3,2,3,\ldots)$ in order to prove $\bar{n}(d,t)<2d+\lceil\frac{d+1}{2}\rceil+t$. Figures \ref{contraexample1} (b) and (c) provide examples of this phenomena for $d=5, t=2$,
and for $d=4, t=3$, respectively. %seria el ejemplo de r=5
In fact, these examples can be generalized in order to
show that this chessboard  {cannot} be used to prove $\bar{n}(d,t)<2d+\lceil\frac{d+1}{2}\rceil+t$
for odd $d\ge5, t\ge2$ and for even $d\ge 4, t\ge3$.

The upper bound given in Theorem \ref{cotageneral} can be improved when $d$ is even.

\begin{theorem}\label{dpar}
 $\bar{n}(t,d)<2d+(t-1)\frac{d}{2}+3$  for any integer $d\ge4$, $d$-even and $t\ge 2$.
 \end{theorem}

This theorem can be proved by making one final tweak to the chessboard defined previously. The latter is a bit technical and requires some extra work in the same flavour as above. This will be done in the Annex.

%%%%%%%%%%%%%%%%%%%%%%%%%%%%%%%%%%%%
\subsection{Proof of Theorem  \ref{boundsforf_o}}

\begin{proof}[Proof of Theorem  \ref{boundsforf_o}.] Recall that $n(t,d)= \bar{n}(t,d)$ and that $H_0(n(t,d),d)=n(t,d)-t$.

\begin{itemize}
\item $d=1, \ n\ge 2$. We clearly have that $H_0(n,1)=2$ since every convex set in dimension $1$ has as support only two vertices.

\item $d=2, \ n\ge 5$. By Equality \eqref{eq:n(2,5)}, we have $n(t,2)=5+t$ for any integer $t\ge 0$. Therefore, by Remark \ref{relationf_0andn(d,t)}, we have $H_0(t+5,2)=5$ for any integer $t\ge 0$ or, equivalently, $H_0(n,2)=5$ for any integer $n\ge 5$.

\item $d=3, \ n\ge 7$. By Theorem \ref{dimension3}, we have $n(t,3)\le 7+t$ for any integer $t\ge 0$. Therefore, by Remark \ref{relationf_0andn(d,t)}, we have $H_0(t+7,3)\le 7$ for any integer $t\ge 0$ or, equivalently, $H_0(n,3)\le 7$ for any integer $n\ge 7$.
\smallskip

\item $d\ge 2, \ n\le 2d+1$. By  Proposition \ref{easybounds}, $H_0(n,d)=n.$

\item $d\ge 4, \ n \ge 2d+\lceil\frac{d+1}{2}\rceil.$ By Proposition \ref{easybounds}, $H_0(n,d)< n.$

\item $d\ge 4, \ n\ge 2d+\lceil\frac{d+1}{2}\rceil+1.$ By Theorem \ref{tpequeno} $\bar{n}(1,d)< 2d+\lceil\frac{d+1}{2}\rceil+1$. Therefore, $H_0(n,d) < n-1.$

\item $d\ge 4,  \  2d+3+l(d-2) \leq n < 2d+3+(l+1)(d-2),$ $l\geq 1.$  By Theorem \ref{cotageneral}, $\bar{n}(d,t)<2d+(t-1)(d-2)+3$  for every $t\ge 2.$ Since $H_0(n(d,t),d)=\bar{n}(t,d)-t$ then, for a given $n$, it would be enough to work out $t$ such that $2d+(t-2)(d-2)+3 \leq {n}<2d+(t-1)(d-2)+3,$ in order to conclude $H_0(n,d) < n-t.$

It is not hard to see that by taking $t=\lfloor\frac{n-2d-3}{d-2}\rfloor+1\ge2$ we obtain that $H_0(n,d)\le n-(\lfloor\frac{n-2d-3}{d-2}\rfloor+2)$ for every $n\ge3d+1$. This can be expressed as; if $2d+3+l(d-2) \leq n < 2d+3+(l+1)(d-2)$ for some $l\ge1$, then  $H_0(n,d) \le n-(l+2)$.
\end{itemize}
\end{proof}

\begin{proof}[Proof of Corollary \ref{boundsforf_d-1}] The desired inequalities are obtained by combining inequalities \eqref{eq:upperbound} and \eqref{upperboundcyclicp} and the values and upper bounds given in Propositions \ref{easybounds} and \ref{easybounds1} and Theorem \ref{boundsforf_o}.

The lower bound is obtained by combining Equality \eqref{form-satcked} and Inequality \eqref{lowerboundHd-1}.

%\textcolor{blue}{The lower bound is obtained by combining Equalities \eqref{form-satcked}, \eqref{upperboundcyclicp} and Proposition \ref{easybounds}.}

\end{proof}

%%%%%%%%%%%%%%%%%%%%%%%%%%%%%%%%%%%
%%%%%%%%%%%%%%%%%%%%%%%%%%%%%%%%%%%
\section{Minimal Radon partitions}\label{sec:RadonPart}

In order to prove Theorem \ref{lem:radon} we need to take a geometric detour on the relationship between faces of convex polytopes, simplices embracing the origin and Radon partitions. There is an old tradition of using Gale transforms to study facets of convex polytopes \cite{GRUN2003} by studying simplices embracing the origin. This equivalence was further extended by Larman \cite{LARMAN72} to studying Radon partitions of points in space.

A \emph{Gale transform} $\bar X$ of a finite set of points $X=\{x_{1}, \ldots,x_{n}\}\subset\R^d$ such that the dimension of their affine span is $r$ is defined as $\bar X =\{ \bar x_{j}= (\alpha_{j, 1}, \ldots \alpha_{j, n-r-1})\}_{j=1}^{n}$, where $\{a_{i}= (\alpha_{1, i},\ldots,\alpha_{n, i})\}_{i=1}^{n-r-1}$ is a basis of the $(n-r-1)$-dimensional space of \emph{affine dependences} of $X$, $D(X)=\{ \alpha=(\alpha_{1},\ldots,\alpha_{n}) | \; \sum_{i=1}^{n}\alpha_{i} x_{i}=0,\; \sum_{i=1}^{n}\alpha_{i}=0 \}$. It is emphasized that $\bar{X}$ is \emph{a} Gale transform of $X,$ rather than \emph{the} Gale transform of $X,$ because the resulting points depend on the specific choice of basis for $D(X).$ Still, different Gale transforms of the same set of points are linearly equivalent.\cite{GRUN2003}

A \emph{Gale diagram} $\hat{X}$ of $X$ is a set of points in $\mathbb{S}^{n-r-2}$ obtained by \emph{normalizing} a Gale transform, that is:
$\hat{X}=\{\hat{x_i}= \frac{\bar{x}_i}{\lVert \bar{x}_i \rVert} |  \bar{x}_i \in \bar{X}, \bar{x}_i\neq 0 \} \cup \{\hat{x_i}= \bar{x}_i |  \bar{x}_i \in \bar{X}, \bar{x}_i= 0 \}.$

\begin{remark} \label{rem:Gale} Let ${X}=\{{x_{1}}, \ldots, {x_{n}}\}$ be a set of $n$ points in $\mathbb{R}^{d}$ and let $\hat{X}$ (respectively. $\bar X$) be its Gale diagram (respectively. Gale transform), then the following statements hold.

\begin{enumerate}[(a)]
\item The $n$ points of $X$ are in general position in $\mathbb{R}^{d}$ if and only if the n-tuple $\hat{X}$ ($\bar X$)
consists of $n$ points in linearly general position in $\mathbb{R}^{n-d-1} .$
\item Faces of $conv(X)$ are in one-to-one correspondence with simplices of $\hat{X}$ ($\bar {X}$) that contain $0$ in their convex hull. More precisely, $Y \subset X$ is a face of $conv(X)$ iff $0 \in relint \: conv(\hat{X}\setminus \hat{Y})$ ($0 \in relint \: conv(\bar{X}\setminus \bar{Y})$).
\item $X$ is projectively equivalent to a set of points $Y$ (by a permissible projective transformation) if and only if there is a  {non-zero} vector $\epsilon=(\epsilon_{1},\ldots, \epsilon_{n}) \in \{1,-1\}^{n}$ ($\lambda=(\lambda_{1},\ldots, \lambda_{n}) \in \mathbb{R}^{n}$) such that $\hat{y}_i=\epsilon_i \hat{x}_i$ ($\bar{y}_i=\lambda_{i}\bar{x}_i$).
\end{enumerate}
We refer the reader to \cite{GRUN2003} for the proofs of this remark.
\end{remark}

Let $X=\{x_1,\dots,x_n\}\subset\R^d$ be a set of points in general position where $n\geq d+2$, and $X= \mathcal{A} \cup \mathcal{B}$ be a disjoint partition of $X$. We define a \emph{partitioned affine projection} (PAP) $\tilde{X}$ of $X$ into the unit $d$-sphere as follows; $\tilde{X}= \{\tilde{x}_i= \mathbf{I}(x_i) \frac{(x_i; 1)}{\lVert (x_i; 1) \rVert} | x_i \in X \} \subset \mathbb{S}^{d},$ where $\mathbf{I}(x_i)=1$ if $x_i \in A$,  $\mathbf{I}(x_i)=-1$ when $x_i \in B,$ and $(x_i; 1)$ is the $d+1$ dimensional vector whose first $d$ entries are identical to those of $x_i$ and last entry is $1$.

\begin{remark} \label{rem:Radon}
$A, B$ is a Radon partition of $X$ if and only if $0 \in conv(\tilde{A}\cup \tilde{B}).$
The proof follows straightforwardly, using linear algebra.
\end{remark}

\begin{proof}[Proof of Theorem  \ref{lem:radon}.]
By Remark \ref{rem:Radon}, we know that if we consider the PAP of $X$ into $\mathbb{S}^{d},$ $\tilde{X},$ we have that $conv(A\cap S) \cap conv(B \cap S) \neq \emptyset$  if and only if $0 \in conv( \tilde{S})$.

Let $\tilde{X}_\epsilon=\{\epsilon_1 \tilde{x}_1, \ldots, \epsilon_n \tilde{x}_n\}$ for $\epsilon=(\epsilon_1, \ldots, \epsilon_n) \in \{1, -1\}^n$. We  define $\rho(\tilde{X}_\epsilon)$  as the number of subsets $\tilde{S}\subset \tilde{X}_\epsilon$  such that $0 \in conv(\tilde{S})$, $\rho(\tilde{X})= \max_{\epsilon \in \{1, -1\}^n} {\rho(\tilde{X}_\epsilon)},$ and $\rho(n,d)=\min_{ \{\tilde{X} \subset \mathbb{S}^d, |\tilde{X}|=n\}}\rho(\tilde{X})$.

It easy to check that $$r(n,d)=\rho(n,d).$$ Notice that $\tilde{X} \subset \mathbb{S}^{d} \subset \mathbb{R}^{d+1}$ while $X\subset\R^d.$

Now, recall that the set $\tilde{X} \subset \mathbb{S}^{d}$ can be considered to be the Gale diagram of a set of points in $X' \subset \R^{n-d-2}$  where each  $\tilde{X}_\epsilon=\{\epsilon_1 \tilde{x}_1, \ldots, \epsilon_n \tilde{x}_n\}$ corresponds to a permissible projective transformation of $X'.$ Therefore, each $(d+2)$--subset $ \tilde{S} \subset \tilde{X}$ such that $0 \in conv( \tilde{S})$ is in one to one correspondence with a co-facet of $X'$ (i.e. the corresponding set $X'\setminus S'$ is a facet of $X').$ Hence, finding $\rho(n,d)$ is equivalent to finding $H_{d'-1}(n,d')$  where $d'=n-d-2.$
\end{proof}

By combining Theorems \ref{boundsforf_d-1} and \ref{lem:radon}, we obtain that for any $d\ge 1$,

{\scriptsize
\begin{eqnarray}\label{eq:bounds-for-r1}
r(n,d)\left\{\begin{array}{ll}
%=1 & \text{ if } n= d+2, d\ge 1\\
=2 & \text{ if } n= d+3,\\
=5 & \text{ if } n= d+4,\\
\le 10 & \text{ if } n= d+5,\\
\le f_{n-d-3}(C_{n-d-2}(n))& \text{ if }  n\ge 2d+3,\\
\le f_{n-d-3}(C_{n-d-2}(n-1)) & \text{ if }  n \leq \frac{5d+8}{3},\\
\le f_{n-d-3}(C_{n-d-2}(n-2)) & \text{ if } n \leq \frac{5d+6}{3},\\
\le  f_{n-d-3}(C_{n-d-2}(n-(l+2))),\ \  1 \leq l & \text{ if } d+4+ \frac{d-3}{l+2} <  n \leq d+4 +\frac{d-1}{l+1}.\\
%\textcolor{blue}{\le f_{n-d-3}(C_{n-d-2}(n-(l+2))),\ \ 1 \leq l} & \textcolor{blue}{\text{ if } d+2 +\frac{2d-2}{l+3} < n \leq d+2 +\frac{2d-2}{l+2}, \ (n-d-2) \text {-even}.}\\
\end{array}\right.
\end{eqnarray}
}
Moreover, by combining Theorem \ref{lem:radon} with Equations  \eqref{form-satcked}, \eqref{lowerboundHd-1} we have

$$r(n,d)=H_{d'-1}(n,d')\ge f_{d'-1}(P_{d'}(n))=(d'-1)n-(d'+1)(d'-2) \text{ if } n\le 2d'+1, \ d\ge 2$$

where $d'=n-d-2$.

We thus obtain that

\begin{eqnarray}\label{eq:bounds-for-r2}
r(n,d) \begin{array}{ll} \geq (n-d-3)n-(n-d-1)(n-d-4) &  \text{ if } n\ge 2d+3, \ d\ge 2.
\end{array}
\end{eqnarray}

%\textcolor{blue}{By Theorem 2,%Theorem \label{lem:radon}
% $r(n,d)=H_2(n,3)$ if  $n=d+5$. On the other hand $f_2(P_3(n))=2n-4$ by   equation (6). Therefore, $r(n,d)\ge 2n-4$ since $H_{d-1}(n,d)\ge f_{d-1}(P_d(n))$ by equation (8).}

%\textcolor{blue}{Ademas, por la misma razon, se puede ver que $r(n,d) \ge (n-d-3)n-(n-d-1)(n-d-4)$, usando equations \eqref{form-satcked}, \eqref{lowerboundHd-1} and Theorem 2. %Theorem \label{lem:radon}}

Equations \eqref{eq:bounds-for-r1} and \eqref{eq:bounds-for-r2} yield to the following bounds in the case when $d=2$.

%\begin{theorem} \label{lem:upbound_rd2} Let $n$ be an integer greater than or equal to $4$, then
%\begin{eqnarray*}%\label{eq:bounds-for-rd2}
%r(2,n) \left\{
%	\begin{array}{ll}
%	=1 & \text{ if } n= 4,\\
%	=2 & \text{ if } n= 5,\\
%	=5 & \text{ if } n= 6,\\
%	\leq 10 & \text{ if } n= 7,\\
%	\leq \displaystyle\binom{\frac{n}{2}+2}{\frac{n}{2}-2} +\displaystyle\binom{\frac{n}{2}+1}{\frac{n}{2}-3}& \text{ for } n\text{-even}, n \ge 8 \\
%	\leq 2\displaystyle\binom{\lfloor \frac{n}{2} \rfloor+2}{\lfloor \frac{n}{2}\rfloor-2} & \text{ for } n\text{-odd}, n\ge 7.
%	\end{array} \right.
%\end{eqnarray*}	
%\begin{eqnarray*}%\label{eq:bounds-for-rd22}
%r(2,7)  \geq 8
%\end{eqnarray*}
%\end{theorem}

%\begin{proof}
%The cases $n=4,5,6,7$ follow directly from the first four lines in Equation \eqref{eq:bounds-for-r1}. The cases $n$-even where $n\ge 8$ and  $n$-odd where $n\ge 7$ can be obtained using the 5th line of Equation \eqref{eq:bounds-for-r1}, replacing the values taking $n-d-3=n-5$ and $n-d-2=n-4$, and by computing $f_{n-5}(C_{n-4}(n))$. The latter is done by using the following well-known formula for the number of facets of a cyclic polytope:
%\begin{eqnarray*}\label{form-cyclic}
%f_{d-1}(C_{d}(n))= \displaystyle\binom{n-\lceil \frac{d}{2}\rceil}{\lfloor \frac{d}{2}\rfloor} + \displaystyle\binom{n-\lfloor \frac{d}{2}\rfloor-1}{\lceil \frac{d}{2}\rceil -1}.
%\end{eqnarray*}
%The equation $r(2,7)  \geq 8$ follows directly from Equation \ref{eq:bounds-for-r2}.
%\end{proof}

%We can obtain the following values and bounds for $r(d,n)$.

\begin{theorem}\label{equalityd=2} Let $n\geq 4$ be an integer. Then,
$$\begin{array}{rlll}
 & r(5,2) =2,\\
 & r(6,2) =5, \\
  & r(7,2) =10, \\
 2(2n-9)\le   & r(n,2) & \text{ if }  n\ge 7,\\
  & r(n,2) \leq 2\displaystyle\binom{\lfloor \frac{n}{2} \rfloor+2}{\lfloor \frac{n}{2}\rfloor-2} & \text{ if }  n\ge 7, n\text{-odd},\\
2\displaystyle\binom{\lceil \frac{n-1}{2}\rceil+1}{\lceil \frac{n-1}{2}\rceil-3} \leq & r(n,2) \leq \displaystyle \binom{ \frac{n}{2}+2}{ \frac{n}{2}-2} +\displaystyle\binom{\frac{n}{2}+1}{ \frac{n}{2}-3}& \text{ if }  n\ge 8, n\text{-even},\\
& r(n,2) =2\displaystyle\binom{\lceil \frac{n}{2} \rceil+1}{\lceil \frac{n}{2}\rceil-3}  &\text{ if }  n\ge 9, n\text{-odd}.
\end{array}$$
\end{theorem}

\begin{proof}
$\bullet$ If $n=5,6$ then the values are obtained directly from Equation \eqref{eq:bounds-for-r1}.

$\bullet$ If $n=7$ then from \eqref{eq:bounds-for-r1} we have $r(7,2)\le 10$ and from \eqref{eq:bounds-for-r2} we have $r(7,2)\ge (7-5)7-(7-3)(7-6)=14-4=10$, and the equality follows.

$\bullet$ The lower bound for $r(n,2)$ when $n\ge 7$ is a straightforward calculation from \eqref{eq:bounds-for-r2}.
\smallskip

By the forth inequality in \eqref{eq:bounds-for-r1}, we have $r(n,2)\le f_{n-5}(C_{n-4}(n))$ for any $n\ge 7$.
Now by taking $k=d-1$ in the formula \eqref{form-cyclic}, we have

\begin{eqnarray*}\label{form-cyclic1}
f_{d-1}(C_{d}(n))= \displaystyle\binom{n-\lceil \frac{d}{2}\rceil}{\lfloor \frac{d}{2}\rfloor} + \displaystyle\binom{n-\lfloor \frac{d}{2}\rfloor-1}{\lceil \frac{d}{2}\rceil -1}.
\end{eqnarray*}

Therefore, by taking $d=n-4$, we have
\begin{equation}\label{eq:Cn-4}
r(n, 2)\le  \binom{n-\lceil \frac{n-4}{2}\rceil}{\lfloor \frac{n-4}{2}\rfloor} + \displaystyle\binom{n-\lfloor \frac{n-4}{2}\rfloor-1}{\lceil \frac{n-4}{2}\rceil -1}.
\end{equation}

$\bullet$ The upper bounds for the cases $n\ge 8$, $n$-even and $n\ge 7$, $n$-odd are obtained from \eqref{eq:Cn-4}.

Recall that  $\nu(k,d)\ge d + \left\lceil \frac{d}{k} \right\rceil +1$ for $k\ge 2$ (see Equations \eqref{eq:boundsmcmullengen} and \eqref{eq:boundsmcmullengen1}). Therefore, by taking $k=\lfloor \frac{n-4}{2}\rfloor$ and $d=n-4$, we obtain that

\begin{eqnarray}\label{eq:v}
\nu\left(\left\lfloor \frac{n-4}{2} \right\rfloor,n-4\right) \geq \left\{
	\begin{array}{ll}
	n-1 & \text{ if } n-4\text{ even,}\\
	n & \text{ if } n-4\text{ odd,}\\
	\end{array} \right.
\end{eqnarray}	
for any integer $\lfloor \frac{n-4}{2}\rfloor\ge 2$, that is, for any $n\ge 9$, $n$-odd and $n\ge 8$, $n$-even.

%bounds the largest number of points that can be mapped through a projective transformation onto the vertices of a convex $k$-neighbourly polytope.

$\bullet$ If $n\ge 9$, $n$-odd then Equation \eqref{eq:v} implies that $n$ points in $\mathbb{R}^{n-4}$ can always be mapped by a permissible projective transformation onto the vertices of a $\lfloor \frac{n-4}{2} \rfloor$--neighbourly polytope. Since $C_{n-4}(n)$ is  $\lfloor \frac{n-4}{2}\rfloor$-neighbourly then $r(n,2)=H_{n-5}(n,n-4)\ge f_{n-5}(C_{n-4}(n))$ and since $r(n,2)\le f_{n-5}(C_{n-4}(n))$ for any $n\ge 7$ then for $n\ge 9$, $n$-odd we have

$$\begin{array}{ll}
r(n,2) &=f_{n-5}(C_{n-4}(n))\\
& = \displaystyle\binom{n-\lceil \frac{n-4}{2}\rceil}{\lfloor \frac{n-4}{2}\rfloor} + \displaystyle\binom{n-\lfloor \frac{n-4}{2}\rfloor-1}{\lceil \frac{n-4}{2}\rceil -1} \ \ \ (\text{by}\ \eqref{form-cyclic1})\\
&= 2\displaystyle\binom{\lceil \frac{n}{2}\rceil +1}{\lceil \frac{n}{2}\rceil -3}.
\end{array}$$

$\bullet$ If $n\geq 8$, $n$-even then we may apply  \eqref{eq:v} for $n$ odd to the number $n-1$, the result will clearly be a lower bound for $r(n,2)$, that is
$$2\displaystyle\binom{\lceil \frac{n-1}{2}\rceil+1}{\lceil \frac{n-1}{2}\rceil-3} \leq  r(n,2).$$
\end{proof}

Other bounds for specific values of $n$ and $d$ can easily be obtained by using Equations \eqref{eq:bounds-for-r1} and \eqref{eq:bounds-for-r2}, for instance,

$$17\le r(9,3)\le 27.$$
\subsection{Pach and Szegedy's question}\label{sec:anapplication}

We may now prove Theorem \ref{thm:unbalanced}.

\begin{proof}[Proof of Theorem  \ref{thm:unbalanced}.] As a consequence of Theorem \ref{equalityd=2}, we clearly have that $r(n,2)$ is of order $o(n^4)$.

Let $A,B$ be a partition that attains the maximum number of induced minimal Radon partitions for the  set $X$, that is $r(X)=r(n,2)$; let $\tilde{X}_\epsilon \subset \mathbb{S}^2$ be it's corresponding PAP, and $X'\subset \mathbb{R}^{n-4}$ be a point configuration whose Gale diagram is $\tilde{X}_\epsilon$.

By combining Remark \ref{rem:Gale}(b) and Remark \ref{rem:Radon} we have that $X$ can have a partition that attains the maximum number of induced Radon partitions for a set of size $n$ if and only if $X'$ is in the projective class of a neighbourly polytope.

%Recall that Equation \eqref{eq:boundsmcmullengen} bounds the largest number of points that can be mapped through a projective transformation onto the vertices of a convex $k$-neighbourly polytope. In our case, as we want a neighbourly polytope, we want $k=\lfloor \frac{n-4}{2}\rfloor$. Replacing this value in the equation results in
%\begin{eqnarray*}
%\nu\left(n-4,\left\lfloor \frac{n-4}{2} \right\rfloor\right) \geq \left\{
%	\begin{array}{ll}
%	n-1 & \text{ if } n-4\text{ even,}\\
%	n & \text{ if } n-4\text{ odd,}\\
%	\end{array} \right.
%\end{eqnarray*}	
%for any integer $d\ge 4$.

Let $n$ be odd. As explained in Theorem \ref{equalityd=2}, Equation \eqref{eq:v} implies that $n$ points in $\mathbb{R}^{n-4}$ can always be mapped by a permissible projective transformation onto the vertices of a $\lfloor \frac{n-4}{2} \rfloor$--neighbourly polytope. The latter implies that $X'$ is in the projective class of a neighbourly polytope. Using Remark \ref{rem:Gale}(b), the neighbourliness of $X'$ translates in $\tilde{X}_\epsilon \subset \mathbb{S}^2$ having the following property: for every subset $S \subset  \tilde{X}_\epsilon$ such that $|S|\leq \lfloor \frac{n-4}{2}\rfloor$, $0 \in conv(\tilde{X}_\epsilon \setminus S)$. Hence, no plane through the origin $H$ is such that $|H^+\cap \tilde{X}_\epsilon|\geq n-\lfloor\frac{n-4}{2}\rfloor$. Therefore, for all planes through the origin $H$, $|H^+\cap \tilde{X}_\epsilon|< n-\lfloor\frac{n-4}{2}\rfloor$. The latter directly implies that in $X$ both $|A|<\lfloor \frac{n}{2}\rfloor+2$ and $|B|<\lfloor \frac{n}{2}\rfloor+2$, as desired.

The even case follows in a similar fashion, by removing one point $x'$ and applying the calculations above to the remaining odd set of points $X \setminus x'$. We highlight that the resulting partition $A,B$ may not necessarily give a maximum number of induced minimal Radon partitions.
\end{proof}

\subsection{Tolerance result}\label{sec:tolerence}

\begin{proposition}\label{prop:nlambda} Let $t\ge 0$ and $d\ge 1$ be integers. Then,
$$n(t,d)=\max_{ m \in \mathbb{N}} \{ m \:| \: \lambda(t, m-d-1) \leq m\}$$
and
$$\lambda(t,d)=\min_{ m \in \mathbb{N}} \{ m \:| \:  m \leq n(t, m-d-1) \}.$$
\end{proposition}

\begin{proof}
Let $X$ be such that it has a partition into two sets $A$,$B$ and a subset $P\subseteq X$ of cardinality  $\mu-i$, for some $0\le i\le t$, such that $conv(A\setminus y)\cap conv(B\setminus y)\neq\emptyset$ for every  $y\in P$ and $conv(A\setminus y)\cap conv(B\setminus y)=\emptyset$ for every  $y\in X\setminus P$. By Remark \ref{rem:Radon}, we know that if we consider the PAP of $X$ into $\mathbb{S}^{d},$ $\tilde{X},$ we have that $conv(A\setminus y) \cap conv(B \setminus y) \neq \emptyset$  if and only if $0 \in conv( (\tilde{A}\setminus \tilde{y}) \cup (\tilde{B} \setminus \tilde{y}))$.

Now let $\rho(t, d)$ be  the smallest number such that for all sets $\tilde{X}$ of cardinality  $\rho$ in $\mathbb{S}^{d},$ there exists a partition of $\tilde{X}$ into two sets $\tilde{A}$,$\tilde{B}$ and a subset $\tilde{P}\subseteq \tilde{X}$ of cardinality  $\rho-i$, for some $0\le i\le t$, such that  $0 \in conv( (\tilde{A}\setminus \tilde{y}) \cup (\tilde{B} \setminus \tilde{y}))$ for every  $\tilde{y}\in \tilde{P}$ and  $0 \not \in conv( (\tilde{A}\setminus \tilde{y}) \cup (\tilde{B} \setminus \tilde{y}))$ for every  $\tilde{y}\in \tilde{X}\setminus \tilde{P}$, then $\lambda(t,d) = \rho(t,d)$. That is, one can seamlessly go from a tolerant partition to a \emph{tolerant} configuration of points in the sphere.

For the next part we will need to establish a relationship between $n(t, d)$ and $ \rho(t,d)$. This relationship arises from the connection between projective transformations of points and antipodal functions of their Gale diagrams, as has already been explored in Theorem \ref{lem:radon}.

Let $y$ be a point strictly in the interior of $conv(X)$. Recall that if we consider the Gale diagram of $X$, $\hat X \subset \mathbb{S}^{n-d-1}$, by Remark \ref{rem:Gale}.(b), as $p$ is not a face of $conv(X)$, $ 0 \not \in conv(\hat{X} \setminus \hat y)$. Also Remark \ref{rem:Gale}.(c) draws the connection between projective transformations of $X$ and taking diametrically opposite points in $\hat X.$

Thus, if we consider the Gale diagram of a set of $n=n(t,d)$ points $\hat X$, we must have that for some
$\hat{X_\epsilon}=\{\epsilon_1 \hat{x}_1, \ldots, \epsilon_n \hat{x}_n\}$ for
$\epsilon=(\epsilon_1, \ldots, \epsilon_n) \in \{1, -1\}^n$,
there is a set of at most $n-i$ points, for some $0 \leq i \leq t$, $\hat{P_\epsilon}$ such that $0 \in conv( \hat{X}_{\epsilon} \setminus \hat{y})$ for $\hat{y} \in \hat{P}_{\epsilon}$.
Thus $\rho(t, n-d-1) \leq n$, and the necessary partition is given by the signs of the epsilons.

Conversely, let $\hat{X}$ be a set of points $\rho = \rho(t,d')$ points, then the Gale transform of these points, $X$ will be such that there is a set of at most $t$ points, such that they are in the interior of $conv(X)$. This is $\rho \leq n(t, \rho-d'-1)$.

As argued at the beginning of the proof, in both inequalities we can straight forwardly substitute $\rho$ for $\lambda$ obtaining
$$n(t,d)=\max_{ m \in \mathbb{N}} \{ m \:| \: \lambda(t, m-d-1) \leq m\} \text{ and } \lambda(t,d)=\min_{ m \in \mathbb{N}} \{ m \:| \:  m \leq n(t, m-d-1) \}.$$
 as desired.
\end{proof}

This proposition can be considered as a generalization of a result due to Larman \cite{LARMAN72} obtained when $t=0$.

%%%%%%%%%%%%%%%%%%%%%%%%%%%%%%%%%%%%
%%%%%%%%%%%%%%%%%%%%%%%%%%%%%%%%%%%%
\section{Arrangements of (pseudo)hyperplanes}\label{sec:concluding-rem}
The so-called Topological Representation Theorem, due to Folkman and Lawrence \cite{FOLKMAN1978199}, states that loop-free oriented matroids of rank $d+1$ on $n$ elements (up to isomorphism) are in one-to-one correspondence with arrangements of pseudo-hyperplanes in the projective space $\mathbb{P}^{r-1}$ (up to topological equivalence).

%If we denote as $\mathcal{H}=\{h_i\}_{1\le i\le n}$ an arrangement of (pseudo)hyperplanes and $M_{\mathcal H}$ as its corresponding oriented matroid, where  $e_i$ is the element of $M_{\mathcal H}$ corresponding to (pseudo)hyperplane $h_i$, then, it is known that an acyclic reorientation of $M_{\mathcal A}$ having $\{e_{i_1},e_{i_2},\ldots,e_{i_l}\}$, with $l\le n$, as interior elements corresponds to a tope in ${\mathcal H}$ which is bordered precisely by the (pseudo)hyperplanes $h_j\not\in\{h_{i_1},h_{i_2},\ldots,h_{i_l}\}$.

A $d$-arrangement of $n$ pseudo-hyperplanes is called {\em simple} if $n\ge d$  and every intersection of $d$ pseudo-hyperplanes is a unique distinct point. It is known that simple arrangements correspond to uniform oriented matroids. It is well known that a tope corresponds to an acyclic reorientation (projective transformations) having as interior elements precisely those pseudo-hyperplanes not bordering the tope.

By the above discussion, we may redefine $\bar n(t,d)$ in terms of hyperplane arrangements:

\begin{quote}
$\bar n(t,d):=$ the  largest integer $n$ such that any simple arrangement of $n$ (pseudo)hyperplanes in $\mathbb{P}^d$ contains a tope of size at least $m-t$.
\end{quote}

\begin{proposition}\label{corolowerboundvd=2}
Every simple arrangement of at least $5$ pseudo-lines in $\mathbb{P}^2$ has a tope of size at least $5$, that is,
\begin{equation*}
\hbox{$5+t\le {\bar n}(t,2)$ for every integer $t\ge 0$.}
\end{equation*}
\end{proposition}

\begin{proof}
The proof is by induction on the set of $n$ (pseudo) lines. By Equation \eqref{eq:boundsmcmullengen}, any arrangement of $5$ (pseudo) lines in $\mathbb{P}^2$ has a tope of size $5$ and thus the proposition holds for $n=5$. We suppose the result true for $n'<n$ and will prove that any arrangement $H$ of $n\ge6$ (pseudo) lines in $\mathbb{P}^2$ has a tope of size at least $5$.  Let $l\in H$, then by induction $H\setminus {l}$ has a tope $T$ of size at least $5$ in $\mathbb{P}^2$. If $l$ does not touch $T$ then $T$ is a tope of $H$ of size at least $5$ in $\mathbb{P}^2$. Otherwise, $l$ divides $T$ into two topes, and since $H$ is simple then one of these two topes is of size at least $5$.
\end{proof}

Combining Proposition \ref{corolowerboundvd=2} and Theorem \ref{dimension2} for $d+2$, we obtain:
\begin{equation}\label{eq:n(2,5)}
\bar{n}(t,2)=5+t \text{ for any integer } t\ge 0.
\end{equation}

%%%%%%%%%%%%%%%%%%%%%%%%%%%%%%%%%%%%
%%%%%%%%%%%%%%%%%%%%%%%%%%%%%%%%%%%%
%\section{Concluding remarks}

For the case $d=3$, Theorem \ref{dimension3} implies that $\bar{n}(t,3)\le 7+t$ for any integer $t\ge 1$, that is, for any $n\ge 7$ there exists a simple arrangement of $n$ (pseudo)planes in $\mathbb{P}^3$ with every tope of size at most $7$. This supports the following:

\begin{conjecture}\label{questionv(t,3)=t+7}  $\bar{n}(t,3)= 7+t$ for any integer $t\ge 1$.
\end{conjecture}
Furthermore, we may ask the following general questions:

\begin{question}\label{conjeturav(d,t)=2d+t} Let $d\ge2$ and $t\ge 0$ be integers.
Is it true that  $\bar{n}(t,d)=2d+1+t?$  In other words, is it true that any simple arrangement of $n\ge 2d+1$ (pseudo)hyperplanes in $\mathbb{P}^d$ contains a tope of size at least $2d+1$
 and conversely, for any $n\ge 2d+1$ there exists a  simple arrangement of $n$ (pseudo)hyperplanes in $\mathbb{P}^d$ with every tope of size at most $2d+1$ ?
\end{question}
Or, alternatively,

\begin{question}\label{conjeturac(d)} Let $d\ge2$ and $t\ge 0$ be integers.
 Is there a constant $c(d)\ge 1$ such that $\bar{n}(t,d)=2d+1+c(d)t$ ?
\end{question}

%%%%%%%%%%%%%%%%%%%%%%%%%%%%%%%%%%%%
%%%%%%%%%%%%%%%%%%%%%%%%%%%%%%%%%%%%
\bibliographystyle{amsplain}
\bibliography{bib_number_facets}

\providecommand{\bysame}{\leavevmode\hbox to3em{\hrulefill}\thinspace}
\providecommand{\MR}{\relax\ifhmode\unskip\space\fi MR }
% \MRhref is called by the amsart/book/proc definition of \MR.
\providecommand{\MRhref}[2]{%
  \href{http://www.ams.org/mathscinet-getitem?mr=#1}{#2}
}
\providecommand{\href}[2]{#2}
\begin{thebibliography}{10}

\bibitem{BARNETTE1973}
David Barnette, \emph{A proof of the lower bound conjecture for convex
  polytopes}, Pacific J. Math. \textbf{46} (1973), no.~2, 349--354.

\bibitem{BILLERA1981}
Louis~J. Billera and Carl~W. Lee, \emph{A proof of the sufficiency of
  mcmullen's conditions for f-vectors of simplicial convex polytopes}, Journal
  of Combinatorial Theory, Series A \textbf{31} (1981), no.~3, 237 -- 255.

\bibitem{OM1999}
Anders Bjorner, Michel~Las Vergnas, Bernd Sturmfels, Neil White, and Gunter~M.
  Ziegler, \emph{Oriented matroids}, 2 ed., Encyclopedia of Mathematics and its
  Applications, Cambridge University Press, 1999.

\bibitem{CORDOVIL1985157}
Raul Cordovil and Ilda~P. Silva, \emph{A problem of mcmullen on the projective
  equivalences of polytopes}, European Journal of Combinatorics \textbf{6}
  (1985), no.~2, 157 -- 161.

\bibitem{FOLKMAN1978199}
Jon Folkman and Jim Lawrence, \emph{Oriented matroids}, Journal of
  Combinatorial Theory, Series B \textbf{25} (1978), no.~2, 199 -- 236.

\bibitem{NGC2015}
Natalia Garc{\'i}a-Col{\'i}n and David Larman, \emph{Projective equivalences of
  k-neighbourly polytopes}, Graphs and Combinatorics \textbf{31} (2015), no.~5,
  1403--1422.

\bibitem{GarciaRaggiPensado}
Natalia Garc\'ia-Col\'in, Miguel Raggi, and Edgardo Rold\'an-Pensado, \emph{A
  note on the tolerant tverberg theorem}, Discrete Comput. Geom. \textbf{58}
  (2017), no.~3, 746--754.

\bibitem{GRUN2003}
Branko Gr{\"u}nbaum, \emph{Convex polytopes}, 2 ed., Graduate Texts in
  Mathematics, vol. 221, Springer-Verlag New York, 2003.

\bibitem{LARMAN72}
David Larman, \emph{On sets projectively equivalent to the vertices of a convex
  polytope}, Bulletin of the London Mathematical Society \textbf{4} (1972),
  no.~1, 6--12.

\bibitem{MCMULLEN1971187}
Peter McMullen, \emph{On the upper-bound conjecture for convex polytopes},
  Journal of Combinatorial Theory, Series B \textbf{10} (1971), no.~3, 187 --
  200.

\bibitem{MonteRam}
Luis~Pedro Montejano and Jorge L.~Ram{\'i}rez Alfons{\'i}n, \emph{Roudneff's
  conjecture for lawrence oriented matroids}, The Electronic Journal of
  Combinatorics \textbf{22} (2015), no.~2, \#P2.3.

\bibitem{PACH2003}
Janos Pach and Mario Szegedy, \emph{The number of simplices embracing the
  origin}, Discrete Geometry: In Honor of W. Kuperberg's 60th Birthday
  \textbf{6} (2003), no.~2, 381--386.

\bibitem{RAMIREZALFONSIN2001}
Jorge~L. Ram{\'i}rez-Alfons{\'i}n, \emph{Lawrence oriented matroids and a
  problem of mcmullen on projective equivalences of polytopes}, European
  Journal of Combinatorics \textbf{22} (2001), no.~5, 723 -- 731.

\bibitem{Richter}
J.~Richter-Gebert, \emph{Oriented matroids with few mutations, in: Oriented
  matroids}, Discrete Comput. Geom. \textbf{10} (1993), 251--269.

\bibitem{Roudneff1991}
Jean-Pierre Roudneff, \emph{Cells with many facets in arrangements of
  hyperplanes}, Discrete Math. \textbf{98} (1991), 185--191.

\bibitem{Roudneff1988}
Jean-Pierre Roudneff and Bernd Sturmfels, \emph{Simplicial cells in
  arrangements and mutations of oriented matroids}, Geometriae Dedicata
  \textbf{27} (1988), no.~2, 153--170.

\bibitem{Shannon1979}
R.~W. Shannon, \emph{Simplicial cells in arrangements of hyperplanes},
  Geometriae Dedicata \textbf{8} (1979), no.~2, 179--187.

\bibitem{SoberonStrauz}
Pablo Sober\'on and Ricardo Strauz, \emph{A generalisation of tverberg's
  theorem}, Discrete Comput. Geom. \textbf{2012} (47), no.~3, 455--460.

\end{thebibliography}
%%%%%%%%%%%%%%%%%%%%%%%%%%%%%%%%%%%%%%%%
%%%%%%%%%%%%%%%%%%%%%%%%%%%%%%%%%%%%%%%%
\pagebreak
\section*{Annex}\label{annex}
We shall prove Theorem \ref{dpar}. First, we need the following.

\begin{lemma}\label{lemmageneral3}Let $A=A_{r,h(r)}$ be a matrix and suppose that  $TT$ always passes strictly above all corners after  the $1$--st corner. Then the following holds:
\begin{itemize}
 \item [(i)] If $r\ge5$ is odd, $B[A]$ has a sequence $(2,t+3,2,t+1,2,t+1,\ldots,2,t+1)$ for some $t\ge2$, $h(2)=t+3$ and $h(m)=2\lceil\frac{m-1}{2}\rceil+(t+1)\lfloor\frac{m-1}{2}\rfloor+3$ for every $3\le m\le r$, then $a_{i,t+3}\in BT$ for some $i\le 2$. %(see Figure \ref{}).

 \item [(ii)]  If $r\ge4$ is even, $B[A]$ has a sequence $(t+1,2,t+1,2,\ldots,t+1,2,t+1)$ for some $t\ge2$ and $h(r)=2\lfloor\frac{m-1}{2}\rfloor+(t+1)\lceil\frac{m-1}{2}\rceil+1$ for every $1\le m\le r$, then $a_{1,t}\in BT$. Moreover, if $TT$ and $BT$ do not share steps from columns $h(2)$ to $h(r)$, then  $a_{1,t+1}\in BT$.

\item [(iii)] If $r\ge5$ is odd, $B[A]$ has a sequence $(t+1,2,t+1,2,t+1\ldots,2,t+3,2)$ for some $t\ge2$,  $h(2)=t+3$, $h(r)=2\lfloor\frac{m-1}{2}\rfloor+(t+1)\lceil\frac{m-1}{2}\rceil+1$ for every $3\le m\le r$ and $TT$ and $BT$ do not share steps from columns $h(2)$ to $h(r)$, then $a_{i,t+3}\in BT$ for some $i\le 2$.
\end{itemize}
\end{lemma}
\begin{proof}
 (i) As $2r-h(r-1)+h(2)-3=2r-(r-1+\frac{(t+1)(r-3)}{2}+3)+t=\frac{2r+2t-4-(t+1)(r-3)}{2}\le \frac{9-r}{2}\le2$ when $r\ge5$ is odd and $t\ge2$. Then   $a_{i,h(2)}\in BT$  for some $i\le 2$ by Lemma   \ref{lemmageneral1} (i).

(ii) As $2r-h(r-1)+h(2)-3=2r-(r-2+(t+1)\frac{(r-2)}{2}+1)+t-1\le 5-\frac{r}{2}\le3$ when $r\ge4$ is even and $t\ge2$, we obtain by Lemma   \ref{lemmageneral1} (i) that  $a_{i,h(2)}\in BT$  for some $i\le 3$.  Then, as $a_{1,i}\in PT$ for $i\le t+3$, we conclude by the rules of  {Proposition}  \ref{prop:chessboard} that $a_{1,t}\in BT$. If $TT$ and $BT$ do not share steps from columns $h(2)$ to $h(r)$, as   $2r-h(r)+h(2)-1=2r-(r-2+ (t+1)\frac{(r)}{2}+1)+t+1\le4-\frac{(r)}{2}\le2$ when $r\ge4$ is even and $t\ge2$, we obtain by Lemma   \ref{lemmageneral1} (i) that  $a_{i,h(2)}\in BT$  for some $i\le 2$.  Then, as $a_{1,i}\in PT$ for $i\le t+3$, we conclude by the rules of  {Proposition} \ref{prop:chessboard} that $a_{1,t+1}\in BT$.

(iv) As $2r-h(r)+h(2)-1=2r-(r-1+\frac{(t+1)(r-1)}{2}+3)+t+2=\frac{2r+2t-(t+1)(r-1)}{2}\le \frac{7-r}{2}\le1$ when $r\ge5$ is odd and $t\ge2$. Then  $a_{i,h(2)}\in BT$  for some $i\le 2$ by Lemma \ref{lemmageneral1} (ii).
\end{proof}

We will use the following remark.

\begin{remark}\label{obssequences}
Let $B[A_1]$ and $B[A_2]$  be with sequences $(x_1,x_2,\ldots,x_{r-1})$ and $(y_1,y_2,\ldots,y_{r-1})$,  $y_i\ge x_i$, $1\le i\le r-1$, respectively. If for any  plain travel $PT$ in $A_1$, the corresponding $\mathcal{A}_1$ in which $PT$ is the Top Travel  has $k$ interior elements, then for any  plain travel $PT$ in $A_2$, the corresponding $\mathcal{A}_2$ in which $PT$ is the Top Travel  has $k$ interior elements.
\end{remark}

 We may now prove Theorem \ref{dpar}.

\begin{proof}[Proof of Theorem \ref{dpar}]
Let $A=A_{r,h(r)}$  be a matrix  where $h(r)$ is defined as %$h(1)=1$,
$h(2)=t+3$, $h(m)=2\lceil\frac{m-1}{2}\rceil+(t+1)\lfloor\frac{m-1}{2}\rfloor+3$ for every $3\le m\le r$
 and  $B[A]$ has sequence $(2,t+3,2,t+1,2,t+1,\ldots,2,t+1)$ for $t\ge2$. We shall show  by induction on $r$ that for odd $r\ge3$ and for any  plain travel $PT$ in $A$, the corresponding $\mathcal{A}$ in which $PT$ is the Top Travel  has at least $t+1$ interior elements. In particular, as $h(r)=r-1+(t+1)\frac{r-1}{2}+3=2d+(t-1)\frac{d}{2}+3$ for odd $r\ge5$, we will prove the theorem for $d\ge4$ and $t\ge2$.

 For $r=3$, the result follows by Theorem \ref{dimension2}, then assume that the theorem holds for $r-1$ and we show it for odd $r\ge 5$. Suppose that  the $m$--st corner is the last corner that $PT$ meets in $A$, for some  $1\le m\le r-1$.  If  $PT$ always passes strictly below the $i$--st corner for  $i>m$, then there would be at least $t+2$ interior elements in $\mathcal{A}$ (from columns $h(r-1)$ to $h(r)$). Hence,  we may suppose that $PT$ always passes strictly above the $i$--st corner for  $i>m$.

The case $m=1$ holds by  Lemma \ref{lemmageneral3} (i). The case $m=2$ holds applying Lemma \ref{lemmageneral2} (ii) on sub-matrix $\mathcal{A}^-_m$ and then Remark \ref{dim1}. The case $m=3$ holds applying Lemma \ref{lemmageneral2} (i) on sub-matrix $\mathcal{A}^-_m$ and then induction hypothesis on $\mathcal{A}^+_m$. We have the following cases.

 \emph{Case $m$ odd and $5\le m\le r-2$ .} Applying  Lemma \ref{lemmageneral2} (i) on $\mathcal{A}^-_m$, $a_{m,h(m)},a_{m,h(m)+1}\in BT$, or  column $h(r)$ is an interior element and $a_{m+1,h(m)}, a_{m,h(m)}\in BT$.  If $a_{m,h(m)},a_{m,h(m)+1}\in BT$, the theorem holds by induction hypothesis on $\mathcal{A}^+_m$. If column $h(r)$ is an interior element and $a_{m+1,h(m)}, a_{m,h(m)}\in BT$, each interior element of $\mathcal{A}^+_m$ is an interior element of $\mathcal{A}$, except for column $h(m)$, then $\mathcal{A}$ has at least $t$ interior elements from columns $1$ to $h(m)$ by induction hypothesis on $\mathcal{A}^+_m$ and one interior element in  column $h(r)$.

\emph{Case $m$ even and $4\le m\le r-3$.}
As $\mathcal{A}^-_m$  has at least $4$ rows,  $a_{m,h(m)},a_{m,h(m)+1}\in BT$ by  Lemma \ref{lemmageneral2} (ii). Then, as each interior element of $\mathcal{A}^+_m$ is an interior element of $\mathcal{A}$, applying Theorem \ref{tpequeno} and Remark \ref{obssequences} to $\mathcal{A}^+_m$, we obtain that $\mathcal{A}$ has at least $2$ interior elements from columns $1$ to $h(m)$.
Suppose first that  $TT$ and $BT$ do not share steps from columns $h(2)$ to $h(r)$. Then, by Lemma \ref{lemmageneral3} (ii), $a_{m,h(m)+t}\in BT$, concluding that  $\mathcal{A}$ has $t-1$ interior elements from columns $h(m)+1$ to $h(m)+t-1$.  Now, suppose that  $TT$ and $BT$ share at least one step. Then,  by Lemma \ref{lemmageneral3} (ii), $a_{m,h(m)+t-1}\in BT$, concluding that  $\mathcal{A}$ has $t-2$ interior elements from columns $h(m)+1$ to $h(m)+t-2$. By the election of $m$,  $TT$ and $BT$ share at least one step from column $h(m)+1$ to $h(r)$, say  $a_{i,j},a_{i,j+1}$ for some $i\le r$ and some $h(m)+1\le j\le h(r)$. If $i=r$, then $\mathcal{A}$ has one interior element in column $j+1$ concluding the proof in this case, then we may suppose that $i<r$. Hence,  the squares in between of $TT$ and $BT$  after column $h(m)$ must be white, concluding by the rules of  {Proposition} \ref{prop:chessboard} that $a_{m,h(m)+t+1}\in BT$ and the theorem holds also in this case.

\emph{Case $m=r-1.$}
First suppose that $BT$ arrives at the $k$-th corner for some  $2\le k\le r-1$. As each interior element of $\mathcal{A}^+_k$ and $\mathcal{A}^-_m$
     is an interior element of $\mathcal{A}$, except for (maybe) columns $h(k)$ and $h(m)$,  $\mathcal{A}$ has at least $t$ interior elements from columns $1$ to $h(k)$ by induction hypothesis on $\mathcal{A}^+_k$ and at least $t-1\ge 1$ interior elements from columns $h(m)$ to $h(r)$ by Remark \ref{dim1} (since $\mathcal{A}^-_m$ is a $2\times (t+2)$ matrix), concluding the proof of this case. Similarly, the proof holds if $BT$ arrives at $a_{i,h(k)}$ for $2\le k\le r-2$ and $i\le k$.
        Now suppose that $BT$ passes always below the $i$--st corner, for every $i\ge2$. In particular, as $BT$ does not arrives at the $m$-th corner, every interior element of $\mathcal{A}^-_m$  is an interior element of $\mathcal{A}$, concluding by Remark \ref{dim1} that $\mathcal{A}$ has $t$ interior elements from columns $h(m)+1$ to $h(r)$. So, we may suppose  that $TT$ and $BT$ do not share steps from columns $1$ to $h(2)$, otherwise the theorem holds.  Also, if $a_{m,h(m)-1}\in PT$, then column $h(m)$ is an interior element and the theorem holds. So, we may suppose that $a_{m,h(m)-1}\not\in PT$. Hence, $a_{r,h(m)-2}\in BT$ by the rules of  {Proposition} \ref{prop:chessboard}.
     Let $\mathcal{A}^{'}$ be the matrix obtained by turning the matrix $\mathcal{A}$ upside down. We observe that $BT$ and $PT$ are the Top and Bottom Travels of $\mathcal{A}^{'}$, respectively. Let define the $i$--st corners of $\mathcal{A}^{'}$ as $a_{r-i+1,h(r-i+1)}$ for $i\neq 2$ and define the
      $2$--st corner of $\mathcal{A}^{'}$ as $a_{m,h(m)-1}$.
     Notice that $BT$ always passes strictly above all corners of  $\mathcal{A}^{'}$ after the $1$--st corner of $\mathcal{A}^{'}$. Moreover, $B[A']$ has the same sequence and the same $2$--st corner as that considered in Lemma \ref{lemmageneral3} (iii). Hence, as $TT$ and $BT$ do not share steps from columns $h(1)$ to $h(m)-1$, we know  by Lemma \ref{lemmageneral3} (iii) that $a_{i,h(m)-1}\in PT$ for some $i\ge r-1$, but this is a contradiction since  we had assumed that $a_{m,h(m)-1}\not\in PT$ and clearly $a_{r,h(m)-1}\not\in PT$, concluding the proof.
\end{proof}

The chessboard considered in Theorem \ref{dpar}  {cannot} be used to prove  $\bar{n}(d,1)<2d+(t-1)\frac{d}{2}+3$   for odd $d$. Figure \ref{contraexample1} (b)  gives an example for $d=5$ and $t=2$,  %seria el ejemplo de r=6
This example can be generalized in order to show that this kind of chessboard  {cannot} be used to prove $\bar{n}(d,1)<2d+(t-1)\frac{d}{2}+3$
for odd $d\ge5$ and $t\ge2$.

 %%%%%%%%% FIGURA %%%%%%%%%%%%%%%%%%%
 \begin{figure}[htb]
\begin{center}
 \includegraphics[width=.9\textwidth]{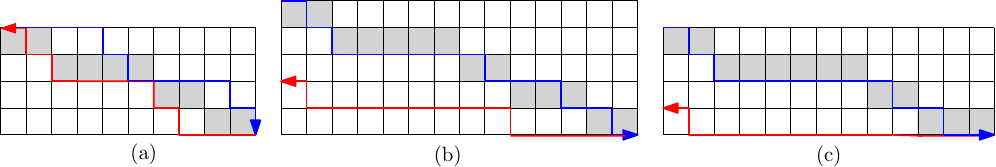}
 \caption{Figures (a), (b) and (c) show the matrices $A_{5,11}$, $A_{6,15}$ and $A_{5,14}$   respectively,
 with chessboards  $(2,t+3,2,2)$ for $t=1$ (a), $(2,t+3,2,3,2)$ for $t=2$ (b) and  $(2,t+3,2,3)$ for $t=3$ (c). We observe that for the pair of Top and Bottom travels described in these matrices, $A_{5,11}$ has only one interior element (column $1$), $A_{6,15}$ has only two interior elements (columns $13$ and $15$) and $A_{5,14}$  has only three interior elements (columns $11,13$ and $14$).}\label{contraexample1}
 \end{center}
\end{figure}

\end{document}